\documentclass[11pt,reqno]{amsart}
\usepackage[T1]{fontenc}
\usepackage[utf8]{inputenc}
\usepackage{lmodern}
\usepackage[a4paper,margin=28mm]{geometry}
\usepackage{amsmath,amssymb,amsthm,mathtools}
\usepackage{microtype,booktabs,array,longtable}
\usepackage{xcolor}
\usepackage{tikz}
\usepackage{bm}
\usetikzlibrary{arrows.meta,positioning}
\usepackage[colorlinks=true,linkcolor=blue!45!black,citecolor=blue!45!black,urlcolor=blue!45!black,breaklinks=true]{hyperref}
\hypersetup{pdftitle={Borcea's 2-variance conjecture},pdfauthor={Teng Zhang},
 pdfsubject={A computer-assisted proof in polynomial geometry},
 pdfkeywords={polynomial zeros, critical points, variance, computer-assisted proof}}
\usepackage{enumitem}
\usepackage{needspace}
\setlist[enumerate]{label=\textup{(\roman*)},leftmargin=*,itemsep=2pt,topsep=4pt}
\newtheorem{thm}{Theorem}[section]
\newtheorem{lem}[thm]{Lemma}
\newtheorem{prop}[thm]{Proposition}
\newtheorem{cor}[thm]{Corollary}
\newtheorem{conjecture}[thm]{Conjecture}
\theoremstyle{definition}

\theoremstyle{remark}
\newtheorem{rem}[thm]{Remark}
\numberwithin{equation}{section}
\newcommand{\C}{\mathbb C}
\newcommand{\R}{\mathbb R}
\newcommand{\dd}{\,\mathrm d}

\newcommand{\Rea}{\operatorname{Re}}
\newcommand{\Ima}{\operatorname{Im}}
\newcommand{\doi}[1]{\href{https://doi.org/#1}{\nolinkurl{#1}}}
\newcommand{\arxiv}[1]{\href{https://arxiv.org/abs/#1}{arXiv:#1}}
\newcommand{\step}[1]{\par\smallskip\noindent\emph{#1}\enspace}
\newcommand{\file}[1]{\texttt{\detokenize{#1}}}

\allowdisplaybreaks[1]
\title[Borcea's 2-variance conjecture]{Borcea's $2$-variance conjecture}
\author{Teng Zhang}
\address{School of Mathematics and Statistics, Xi'an Jiaotong University, Xi'an, China}
\date{}
\begin{document}
\begin{abstract}
In this paper, we prove Borcea's 2-variance conjecture. The Lean 4 formalization of the main results are also provided.
\end{abstract}
\maketitle

\section{Introduction}\label{sec:intro}

Very recently, the geometry of polynomials has witnessed a series of major advances.

\begin{itemize}
   \item Mazur \cite{Maz26} obtained a  computer-verified proof of Sendov's conjecture. Tao \cite{Tao26} subsequently digested and reorganized the argument, substantially streamlining the proof and bringing its underlying mathematical structure to the foreground. In this process, Tao also found the same ideas can prove the stronger Phelps--Rodriguez conjecture \cite{PR72}.

    \item 
    Building on the ideas developed in Tao's treatment of Sendov's conjecture \cite{Tao26}, the author \cite{Zha26b} established a quantitative strengthening of Sendov's conjecture named the \emph{quadratic Tang--Zhang inequality}. 

    \item 
    A counterexample of Smale's mean value conjecture was found in \cite{Ada26}.

\end{itemize}
A closer look at the proofs of Sendov's conjecture and the quadratic Tang--Zhang inequality shows that they follow essentially the same pattern:
\begin{center}
\emph{assuming that a counterexample exists: direct arguments in low degrees, analytic estimates in high degrees, and computer-assisted verification for the remaining intermediate cases.}
\end{center}
This observation naturally suggests that the same general strategy may be effective for several related problems, including Schmeisser's conjecture \cite{Sch77}, Borcea's variance conjecture \cite{KPPSS11} and related generalizations \cite{TZ25,Tao26,Zha26a}. In this paper, we focus on Borcea's variance conjecture.

Let $P$ be a complex polynomial of degree $n\ge 2$ with zeros $z_1,\ldots,z_n$, counted with multiplicity. For $0<p\le \infty$,
define  the $p$-variance of the zero set of $P$ by
\begin{equation*}
 \begin{split}
 \sigma_p(P)&=\inf_{c\in\C}
 \left(\frac1n\sum_{j=1}^n|z_j-c|^p\right)^{\frac{1}{p}}
 \qquad(0< p<\infty),\\
 \sigma_\infty(P)&=\inf_{c\in\C}\max_{1\le j\le n}|z_j-c|.
 \end{split}
\end{equation*} By the monotonicity of normalized \(\ell^p\)-norms, \(\sigma_p(P)\) is nondecreasing in \(p\); that is,
\[
\sigma_p(P)\le \sigma_q(P)
\qquad (0<p<q\le\infty).
\]

Certain values of $\sigma_p(P)$ are of particular importance, see \cite[\S 2]{KPPSS11}.
The quantity $\sigma_1(P)$ is called the \emph{mean deviation} of the zero
set of $P$. For $p=2$, by the basic inner space techniques, the minimizing center is the centroid
$c_P=n^{-1}\sum_{j=1}^n z_j$, and hence
\[
\sigma_2(P)^2
=\frac{1}{n}\sum_{j=1}^n |z_j-c_P|^2.
\]
Thus $\sigma_2(P)^2$ is the variance of the zeros of $P$, while
$\sigma_2(P)$ is their standard deviation. For $p=\infty$,
$\sigma_\infty(P)$ is called the \emph{Chebyshev radius} of the zero
set of $P$; equivalently, it is the radius of the smallest closed disk
containing all the zeros of $P$. Any minimizing center in the definition
of $\sigma_\infty(P)$ is called a \emph{Chebyshev center} of the zero set.

Motivated by his work on Sendov's conjecture \cite{Bor96a,Bor96b,Bor98,Bor06,Bor07}, Borcea proposed a family of \(p\)-variance conjectures in which the spread of the zeros is measured by centered moments rather than by the radius of a containing disk; see \cite[Conjecture~1]{KPPSS11}.
\begin{conjecture}[Borcea \(p\)-variance conjectures]\label{conj:borcea}
Let \(P\) be a complex polynomial of degree \(n\ge 2\) and $1\le p\le \infty$. For every zero \(a\) of \(P\), there exists a critical point \(\zeta\) of \(P\) such that
\[
    |a-\zeta|\le \sigma_p(P).
\]
\end{conjecture}

Since \(\sigma_p(P)\) is nondecreasing in \(p\), the case \(p=1\) is the strongest form of Conjecture~\ref{conj:borcea}. At the other extreme, the limiting case \(p=\infty\), which is the weakest member of the family, is equivalent to Sendov's conjecture.

Conjecture~\ref{conj:borcea} is known in several special cases:
\begin{itemize}
    \item \(n\le 3\) and \(p=1\)
    \cite[\S4]{KPPSS11};

    \item \(P\) has at most three distinct zeros and \(p=2\)
    \cite[Theorem~3.2]{KPPSS11};

    \item all zeros of \(P\) are real and \(p=2\); Borcea
    (see \cite[Theorem~6.1]{KPPSS11}) obtained the stronger result that the radius
    \(\sigma_2(P)\) can be reduced by the factor \(1/\sqrt{n-1}\).
\end{itemize}

Our main result settles the case \(n\ge 2\) and \(p=2\) in Conjecture~\ref{conj:borcea}.

\begin{thm}\label{thm:main}
Let $P$ be a complex polynomial of degree $n\ge2$. For every zero \(a\) of \(P\), there exists a critical point \(\zeta\) of \(P\) such that
\begin{equation}\label{eq:main}
 |a-\zeta|\le\sigma_2(P).
\end{equation}
The constant 1 in \eqref{eq:main} is sharp. 
\end{thm}

As an immediate consequence of Theorem~\ref{thm:main}
(see also \cite[\S6.3]{KPPSS11}), we obtain the following corollary.
Indeed, Malamud \cite{Mal05} related the zeros and critical points of a
polynomial to the spectra of a normal matrix and one of its principal
submatrices.  In the circulant setting, Kushel and Tyaglov
\cite[Theorem~2.1]{KT16} showed that if \(C\) is an \(n\times n\)
circulant matrix with characteristic polynomial \(P\), then the spectrum
of its leading \((n-1)\times(n-1)\) principal submatrix consists precisely
of the critical points of \(P\), counted with multiplicity.

\begin{cor}\label{cor:toeplitz-general}
Assume that \(n\ge3\) and \(a_0,\ldots,a_{n-1}\in\C\). Then the
\((n-1)\times(n-1)\) Toeplitz matrix
\[
B=
\begin{pmatrix}
a_0      & a_1      & \cdots & a_{n-2}\\
a_{n-1}  & a_0      & \cdots & a_{n-3}\\
\vdots   & \vdots   & \ddots & \vdots\\
a_2      & a_3      & \cdots & a_0
\end{pmatrix}
\]
has at least one eigenvalue \(\lambda\) satisfying
\[
\left|
\lambda-\sum_{k=0}^{n-1}a_k
\right|^2
\le
\sum_{k=1}^{n-1}|a_k|^2.
\]
\end{cor}

\medskip
\noindent\textbf{Sketch of the proof.}
We briefly describe the proof of Theorem~\ref{thm:main}.
After normalization, a hypothetical counterexample is described by the
reciprocal critical-point coordinates
\[
 q_j=\frac1{a-\zeta_j},
 \qquad
 \mu=\frac1{n-1}\sum_{j=1}^{n-1}q_j,
 \qquad
 s=\frac1{n-1}\sum_{j=1}^{n-1}|q_j|^2,
\]
with \(|q_j|<1\).  Schoenberg's inequality and several
Cauchy--Schwarz estimates reduce the geometry of the critical points to
quantitative restrictions on \(a\), \(|\mu|\), and \(s\).

We then study
$
 F(t)=\prod_{j=1}^{n-1}(1-atq_j).
$
Two polynomial identities, together with direct product estimates and a
centered interpolation between \(q_j=\mu\) and the actual configuration,
give explicit scalar inequalities that every normalized counterexample
must satisfy.

These inequalities exclude all \(n\ge100001\) by elementary analytic
estimates.  For \(4\le n\le100000\), we first obtain additional product
and coefficient bounds and then reduce the problem to a compact box in
the parameters \((a,|\mu|,s)\).  A finite rational subdivision of this
box is excluded by exact-arithmetic certificates with directed rounding.
Together with the known cases \(n=2,3\), this proves the theorem.

\medskip
\noindent\textbf{Organization of the paper.}
In Section~\ref{sec:notation}, we normalize a hypothetical counterexample
and introduce the reciprocal coordinates and second-moment notation.
Section~\ref{sec:reduction} establishes the basic reciprocal-moment
inequalities, polynomial identities, and product bounds.
In Section~\ref{sec:integral}, we develop the direct and centered integral
estimates used throughout the proof.
Section~\ref{sec:large} excludes all degrees \(n\ge100001\) by purely
analytic estimates.
Sections~\ref{sec:products} and \ref{sec:coefficients} derive the additional
product and coefficient bounds needed for the remaining degrees, including
a separate analytic treatment of configurations near \(|\mu|=1\).
Section~\ref{sec:finite} reduces the remaining problem to a compact
three-parameter box and completes the proof by an exact finite verification.
Finally, Appendix~\ref{app:boxes} gives the mathematical specification of
the interval tests used by the certificates, while
Appendix~\ref{app:arithmetic} describes the exact arithmetic, certificate
format, and reproduction procedure.

\medskip
\noindent\textbf{Acknowledgments and AI tools disclosure.}
The author is grateful to Terence Tao for helpful discussions in the
comments on his \href{https://terrytao.wordpress.com/2026/08/12/a-digestion-of-the-proof-of-sendovs-conjecture/}{blog post}.  He also thanks the developers
and maintainers of the computational tools used in the preparation and
verification of this work.

Teng Zhang is supported by the China Scholarship Council, the Young Elite
Scientists Sponsorship Program for PhD Students of the China Association
for Science and Technology, and the Fundamental Research Funds for the
Central Universities at Xi'an Jiaotong University
(Grant No.~xzy022024045).

ChatGPT was used for English-language editing and proofreading, and as an
exploratory tool for discussing possible approaches to selected parts of
the argument.  The mathematical arguments and proofs in the original manuscript were
independently developed, checked, and written by the author.
The computer-assisted part of the proof is certified by the exact
checkers and certificate files described in
Appendices~\ref{app:boxes} and \ref{app:arithmetic}; the acceptance of
a certificate does not rely on output from AI tools.

The Lean~4 formalization in~\cite{ZhaL} was generated using
OpenAI's Codex and checked by the Lean kernel.

\section{Reduction and notation}\label{sec:notation}

We first normalize the problem and introduce the notation that will be used
throughout the proof. Suppose that \eqref{eq:main} fails at a zero of a polynomial of degree $n\ge4$. The variance is positive and the prescribed zero is simple. Translate the centroid to zero, divide the zeros by $\sigma_2(P)$, rotate the prescribed zero to the nonnegative real axis, and make the polynomial monic. The resulting polynomial satisfies
\begin{equation}\label{eq:normalization}
 \begin{gathered}
 P(z)=(z-a)\prod_{j=1}^{n-1}(z-z_j),\qquad a\ge0,\\
 a+\sum_{j=1}^{n-1} z_j=0,\qquad
 a^2+\sum_{j=1}^{n-1}|z_j|^2=n.
 \end{gathered}
\end{equation}
Let $\zeta_1,\ldots,\zeta_{n-1}$ be its critical points.  

The classical Schoenberg inequality states that  if
\(z_1,\ldots,z_n\) are the zeros of a complex polynomial \(P\),
and \(\zeta_1,\ldots,\zeta_{n-1}\) are its critical points, all counted
with multiplicity, then, provided that the centroid of the zeros is the
origin, i.e.\ \(\sum_{j=1}^n z_j=0\), one has
\begin{equation}\label{eq:schoenberg}
	\sum_{k=1}^{n-1} |\zeta_k|^2 \le \frac{n-2}{n} \sum_{j=1}^n |z_j|^2,
\end{equation}
where equality holds if and only if all zeros are collinear in the complex plane. This was conjectured by    Schoenberg \cite{Sch86}  and confirmed independently by Pereira \cite{Per03} and Malamud \cite{Mal05}. It
was subsequently reproved by Cheung–Ng \cite{CN06} and Kushel–Tyaglov \cite{KT16} using distinct
approaches.  \eqref{eq:schoenberg} has since inspired a series of higher-order generalizations, collectively referred to as \emph{Schoenberg type inequalities}, see \cite{BS99,CN06,KT16,LXZ21,GGTW25,Tan25,TZ25, Zha26c, Zha26d}. As an interesting application of \eqref{eq:schoenberg}, the author~\cite{Zha26a,Zha27} presented some improvements of Schmeisser's conjecture \cite{Sch77} for the barycenter case.

Coefficient comparison and the classical Schoenberg inequality \eqref{eq:schoenberg}  give
\begin{equation}\label{eq:critical-moments}
 \frac1{n-1}\sum_{j=1}^{n-1}\zeta_j=0,\qquad
 V_*:=\frac1{n-1}\sum_{j=1}^{n-1}|\zeta_j|^2\le V_n:=1-\frac1{n-1}.
\end{equation}
A counterexample has $|a-\zeta_j|>1$ for every $j$, so \eqref{eq:critical-moments} excludes $a=0$. The Cauchy--Schwarz inequality in \eqref{eq:normalization} gives $a^2\le n-1$.

In the rest of the paper, a \emph{normalized counterexample} means a polynomial satisfying \eqref{eq:normalization} with $a>0$ and $|a-\zeta_j|>1$ for every $j$. Unless stated otherwise, the following notation refers to such a counterexample.

\medskip
\noindent\textbf{Notation.}
We shall use the following notation throughout the argument. For the quantities already appearing in Tao's
exposition \cite{Tao26}, we keep the same notation, namely $a$, $z_j$,
$\zeta_j$, $q_j$, $x$, $y$, $F(t)$, and $J$. We also retain from
\cite{Zha26b} the second-moment notation $s$, $w_j$, $v$, and $\beta_s(t)$.
The remaining quantities \(\mu\), \(V\), \(\rho\), \(e_k\), and \(u_+\) are introduced for the purposes of the present argument.

\begin{itemize}[leftmargin=*,itemsep=4pt,topsep=5pt]

\item $q_j:=\dfrac{1}{a-\zeta_j}$, where $0<|q_j|<1$ for
$1\le j\le n-1$.

\item $\mu:=\dfrac{1}{n-1}\sum_{j=1}^{n-1}q_j=x+iy$, where $x,y\in\R$.

\item $w_j:=q_j-\mu$ for $1\le j\le n-1$. Then
$\sum_{j=1}^{n-1}w_j=0$.

\item $s:=\dfrac{1}{n-1}\sum_{j=1}^{n-1}|q_j|^2$, where $0<s<1$.

\item $v:=\dfrac{1}{n-1}\sum_{j=1}^{n-1}|w_j|^2=s-|\mu|^2$.

\item $V$ denotes a specified upper bound for $V_*$ with $0\le V\le1$.
Its default value is $V_n$ from \eqref{eq:critical-moments}; individual
arguments may use $V=1$ or a sharper bound.

\item $\rho:=2ax-a^2s$.

\item $\beta_s(t)=1-2axt+a^2st^2$, where $0\le t\le1$.

\item $F(t):=\prod_{j=1}^{n-1}(1-atq_j)$.

\item $J:=\prod_{j=1}^{n-1}z_jq_j$.

\item $e_k(\bm \xi):=\sum_{1\le j_1<\cdots<j_k\le n-1}
\xi_{j_1}\cdots \xi_{j_k}$ for $\bm \xi=(\xi_1,\ldots,\xi_{n-1})\in\mathbb C^{n-1}$, and $e_0(\bm \xi):=1$.
Empty sums are zero and empty products are one.

\item $u_+:=\max\{u,0\}$ for $u\in\R$.

\end{itemize}

\section{Reciprocal moments and polynomial identities}\label{sec:reduction}

\begin{lem}\label{lem:moments}
Every normalized counterexample satisfies
\begin{align}
 \rho&\ge |\mu|^2+(1-V)(1-s),&
 |\mu-a|^2&\le(a^2+V-1)(1-s),\label{eq:moment-a}\\
 |1-a\mu|^2&\le Vv,&
 s(a^2+V)&\ge1,\label{eq:moment-b}\\
 \beta_s(1)&\le(a^2+V)v.&&\label{eq:moment-end}
\end{align}
\end{lem}
\begin{proof}
Since $q_j=(a-\overline{\zeta_j})|q_j|^2$ and $\sum_{j=1}^{n-1}\zeta_j=0$,
$$\mu-as=\frac{1}{n-1}\sum_{j=1}^{n-1}\overline{\zeta_j}(1-|q_j|^2).$$
Applying the  Cauchy--Schwarz inequality, we obtain
\[
\begin{aligned}
|\mu-as|^2
&\le
\left(
\frac1{n-1}\sum_{j=1}^{n-1}
|\zeta_j|^2(1-|q_j|^2)
\right)
\left(
\frac1{n-1}\sum_{j=1}^{n-1}
(1-|q_j|^2)
\right)\\
&=
\bigl(V_*-\beta_s(1)\bigr)(1-s)\\
&\le
\bigl(V-\beta_s(1)\bigr)(1-s),
\end{aligned}
\]
which is equivalent to the first inequality in \eqref{eq:moment-a}. 

Since $
|\mu-a|^2
=|\mu|^2-\rho+a^2(1-s)$,
the first inequality in \eqref{eq:moment-a} yields
\[
|\mu-a|^2
\le (a^2+V-1)(1-s),
\]
which proves the second inequality in \eqref{eq:moment-a}.

Since $(a-\zeta_j)q_j=1$ and
$\sum_{j=1}^{n-1}\zeta_j=0$, we have
\[
a\mu-1
=
\frac1{n-1}\sum_{j=1}^{n-1}\zeta_j(q_j-\mu).
\]
Hence, by the Cauchy--Schwarz inequality,
\[
\begin{aligned}
|1-a\mu|^2
&\le
\left(
\frac1{n-1}\sum_{j=1}^{n-1}|\zeta_j|^2
\right)
\left(
\frac1{n-1}\sum_{j=1}^{n-1}|q_j-\mu|^2
\right)\\
&=V_*v
\le Vv,
\end{aligned}
\] which proves the first inequality in \eqref{eq:moment-b}. 

Since $q_j(a-\zeta_j)=1$, the Cauchy--Schwarz inequality gives
\[
\begin{aligned}
1
&=
\left|
\frac1{n-1}\sum_{j=1}^{n-1}q_j(a-\zeta_j)
\right|^2\\
&\le
\left(
\frac1{n-1}\sum_{j=1}^{n-1}|q_j|^2
\right)
\left(
\frac1{n-1}\sum_{j=1}^{n-1}|a-\zeta_j|^2
\right)\\
&=
s(a^2+V_*)
\qquad\text{since }\sum_{j=1}^{n-1}\zeta_j=0\\
&\le s(a^2+V),
\end{aligned}
\]
which proves the second inequality in \eqref{eq:moment-b}.

Finally, since $v=s-|\mu|^2$ and $\mu=x+iy$,
\[
\begin{aligned}
\beta_s(1)
&=1-2ax+a^2s\\
&=|1-a\mu|^2+a^2v\\
&\le (a^2+V)v
\qquad\text{by the first inequality in \eqref{eq:moment-b}},
\end{aligned}
\]
which proves \eqref{eq:moment-end}.
\end{proof}

\begin{lem}\label{lem:identities}
For every normalized counterexample, define
\[
 \mathcal R
 :=
 a^2\int_0^1 t
 \sum_{j=1}^{n-1}q_j^2
 \prod_{\substack{1\le k\le n-1\\k\ne j}}
 (1-atq_k)\dd t.
\] Then
\begin{equation}\label{eq:origin}
(-1)^{n-1}J
=
n\int_0^1F(t)\dd t.
\end{equation}
\begin{equation}\label{eq:differential}
F(1)+(n-1)a\mu\int_0^1F(t)\dd t+\mathcal R
=
1.
\end{equation}
Moreover,
\begin{equation}\label{eq:center-identity}
(-1)^{n-1}a\mu J+(1-a\mu)^n
-na\mu\int_0^1
\bigl(F(t)-(1-a\mu t)^{n-1}\bigr)\dd t=1.
\end{equation}
\end{lem}

\begin{proof}
For the identity \eqref{eq:origin}, see
\cite[Lemma~6(iii), equation~(9)]{Tao26} or
\cite[Lemma~3.1(i)]{Zha26b}.

For the differential identity \eqref{eq:differential}, see
\cite[Lemma~3.1(iii)]{Zha26b}, with \(x+iy=\mu\) and \(R=\mathcal R\).

It remains only to prove \eqref{eq:center-identity}.  Since
\[
na\mu\int_0^1(1-a\mu t)^{n-1}\dd t
=
1-(1-a\mu)^n,
\]
where the identity remains valid when \(a\mu=0\), \eqref{eq:origin} gives
\[
\begin{aligned}
(-1)^{n-1}a\mu J
&=
na\mu\int_0^1F(t)\dd t\\
&=
na\mu\int_0^1
\bigl(F(t)-(1-a\mu t)^{n-1}\bigr)\dd t
+1-(1-a\mu)^n.
\end{aligned}
\]
Rearranging yields \eqref{eq:center-identity}.
\end{proof}

We shall use both a root product bound and a critical-point product bound.

\begin{lem}\label{lem:basic-product}
Define
\[
 \Phi_{n-1}(a)=a\left(1+\frac{1-a^2}{n-1}\right)^{(n-1)/2},
 \qquad \psi(a)=a\exp\left(\frac{1-a^2}{2}\right).
\]
Every normalized counterexample satisfies
\begin{equation}\label{eq:core}
|a\mu J|
\le
\Phi_{n-1}(a)|\mu|\prod_{j=1}^{n-1}|q_j|,\qquad
\text{where }
\Phi_{n-1}(a)\le\min\{1,\psi(a)\}.
\end{equation}

\begin{equation}\label{eq:product-basic}
\left[
\frac{n}{n-1}(1+a^2)s
\right]^{n-1}
\ge n^2.
\end{equation}
For fixed $a$, $\Phi_{n-1}(a)$ is nondecreasing in $n$ on its domain $n-1\ge a^2$. For fixed $n$, it increases up to $a=1$ and decreases thereafter.
\end{lem}
\begin{proof}
We first prove \eqref{eq:core}.  By the normalization
\eqref{eq:normalization},
$
\sum_{j=1}^{n-1}|z_j|^2=n-a^2.
$
Hence the arithmetic--geometric mean inequality gives
\[
\prod_{j=1}^{n-1}|z_j|
\le
\left(
\frac{1}{n-1}\sum_{j=1}^{n-1}|z_j|^2
\right)^{(n-1)/2}
=
\left(\frac{n-a^2}{n-1}\right)^{(n-1)/2}.
\]
Multiplying by \(a\), we obtain
\begin{equation}\label{eq:root-product-upper}
a\prod_{j=1}^{n-1}|z_j|
\le
a\left(\frac{n-a^2}{n-1}\right)^{(n-1)/2}
=
a\left(1+\frac{1-a^2}{n-1}\right)^{(n-1)/2}
=
\Phi_{n-1}(a).
\end{equation}

We next bound \(\Phi_{n-1}(a)\).  Applying the arithmetic--geometric
mean inequality to the \(n\) nonnegative numbers
\[
a^2,\quad
\frac{n-a^2}{n-1},\quad\ldots,\quad
\frac{n-a^2}{n-1}
\]
gives
\[
a^2
\left(\frac{n-a^2}{n-1}\right)^{n-1}
\le1.
\]
The left-hand side is precisely \(\Phi_{n-1}(a)^2\), and hence
$
\Phi_{n-1}(a)\le1.
$

On the other hand, since
$
1+\frac{1-a^2}{n-1}>0,
$
the elementary inequality \(\log(1+u)\le u\) when \(u>-1\)
gives
\[
\begin{aligned}
\log\Phi_{n-1}(a)
&=
\log a+
\frac{n-1}{2}
\log\left(1+\frac{1-a^2}{n-1}\right)\\
&\le
\log a+\frac{1-a^2}{2}
=
\log\psi(a).
\end{aligned}
\]
Thus
$
\Phi_{n-1}(a)\le\psi(a),
$
and consequently
\[
\Phi_{n-1}(a)\le\min\{1,\psi(a)\}.
\]
This gives the second inequality in \eqref{eq:core}.

Finally, by the definition of \(J\),
\[
|a\mu J|
=
a|\mu|
\prod_{j=1}^{n-1}|z_j|
\prod_{j=1}^{n-1}|q_j|.
\]
Combining this identity with \eqref{eq:root-product-upper} yields
\[
|a\mu J|
\le
\Phi_{n-1}(a)|\mu|
\prod_{j=1}^{n-1}|q_j|,
\]
which proves the first inequality in \eqref{eq:core}.

We now prove \eqref{eq:product-basic}.  Using \(\sum_{j=1}^{n-1}z_j=-a\) and
\(\sum_{j=1}^{n-1}|z_j|^2=n-a^2\), we obtain
\[
\sum_{j=1}^{n-1}|a-z_j|^2=
n(1+a^2).
\]
Thus
\[
\frac1{n-1}\sum_{j=1}^{n-1}|a-z_j|^2
=
\frac{n(1+a^2)}{n-1}.
\]
The arithmetic--geometric mean inequality gives
\begin{equation}\label{eq:u-product-upper}
\prod_{j=1}^{n-1}|a-z_j|^2
\le
\left(
\frac{n(1+a^2)}{n-1}
\right)^{n-1}.
\end{equation}

Recall
$
P(z)=(z-a)\prod_{j=1}^{n-1}(z-z_j).
$
Then
\[
P'(a)=\prod_{j=1}^{n-1}(a-z_j).
\]
On the other hand,
$
P'(z)=n\prod_{j=1}^{n-1}(z-\zeta_j),
$
and therefore, using \(q_j=(a-\zeta_j)^{-1}\),
\[
|P'(a)|^2
=
n^2\prod_{j=1}^{n-1}|a-\zeta_j|^2
=
\frac{n^2}{\prod_{j=1}^{n-1}|q_j|^2}.
\]
It follows that
\begin{equation}\label{eq:u-q-product}
\prod_{j=1}^{n-1}|a-z_j|^2
=
\frac{n^2}{\prod_{j=1}^{n-1}|q_j|^2}.
\end{equation}
Moreover, by the arithmetic--geometric mean inequality and the
definition of \(s\),
\begin{equation}\label{eq:q-product-upper}
\prod_{j=1}^{n-1}|q_j|^2
\le
\left(
\frac1{n-1}\sum_{j=1}^{n-1}|q_j|^2
\right)^{n-1}
=
s^{n-1}.
\end{equation}
Combining \eqref{eq:u-product-upper}, \eqref{eq:u-q-product}, and
\eqref{eq:q-product-upper} gives
\[
\frac{n^2}{s^{n-1}}
\le
\frac{n^2}{\prod_{j=1}^{n-1}|q_j|^2}
=
\prod_{j=1}^{n-1}|a-z_j|^2
\le
\left(
\frac{n(1+a^2)}{n-1}
\right)^{n-1}.
\]
Therefore
\[
\left[
\frac{n}{n-1}(1+a^2)s
\right]^{n-1}
\ge n^2,
\]
which proves \eqref{eq:product-basic}.

It remains to verify the two monotonicity assertions.  Fix \(a>0\)
and write
$
m=n-1,
$ and $
b=1-a^2.
$
Then
\[
\log\Phi_m(a)
=
\log a+\frac m2\log\left(1+\frac bm\right).
\]
Regarding \(m\) as a continuous variable, differentiation gives
\[
\frac{\partial}{\partial m}\log\Phi_m(a)
=
\frac12
\left[
\log\left(1+\frac bm\right)
-
\frac{b/m}{1+b/m}
\right].
\]
Notice that for every \(u>-1\), 
\[
h(u)=\log(1+u)-\frac{u}{1+u}\ge 0.
\]
Thus
\[
\frac{\partial}{\partial m}\log\Phi_m(a)\ge0,
\]
and hence, for fixed \(a\), \(\Phi_{n-1}(a)\) is nondecreasing in
\(n\) throughout its domain.

Finally, for fixed \(n\),
\[
\frac{\dd}{\dd a}\log\Phi_{n-1}(a)
=
\frac{n(1-a^2)}{a(n-a^2)}.
\]
Since \(a>0\) and \(a^2\le n-1<n\), the denominator is positive.
Thus the derivative has the sign of \(1-a^2\).  Therefore
\(\Phi_{n-1}(a)\) is increasing for \(0<a\le1\) and decreasing for
\(1\le a\le\sqrt{n-1}\), as claimed.
\end{proof}

\section{Integral estimates}\label{sec:integral}

The estimates in this section will be used twice: first with simple exponential majorants, and then on rational intervals. Write
\begin{equation}\label{eq:constants}
 c_n=\left(\frac{n-1}{n-2}\right)^{(n-2)/2},\quad
 \widetilde c_n=c_n^2,\qquad
 d_n=\left(\frac{n-1}{n-3}\right)^{(n-3)/2},\quad
 \widetilde d_n=d_n^2.
\end{equation}
For $n\ge4$, one has $c_n<5/3$, $\widetilde c_n<3$, $d_n<3$, and $\widetilde d_n<9$.

\subsection{Quadratic and linear majorants}

For the second moment, let \(b:[0,1]\to[0,\infty)\) be any convex function such that
\[
\beta_s(t)\le b(t)
\qquad (0\le t\le1).
\]
Since
\(\beta_s(t)=(n-1)^{-1}\sum_{j=1}^{n-1}|1-atq_j|^2\),
we have
\begin{equation}\label{eq:quadratic-majorant}
\frac1{n-1}\sum_{j=1}^{n-1}|1-atq_j|^2
\le b(t).
\end{equation}

\medskip
\noindent\textbf{Our goal.}
We need to seek a simple linear function \(\ell(t)\) such that
\begin{equation}\label{eq:linear-majorant-aim}
\frac1{n-1}\sum_{j=1}^{n-1}|1-atq_j|
\le \ell(t),
\qquad 0\le t\le1.
\end{equation}

\step{Step 1. Estimate the endpoint \(t=1\).}
By the Cauchy--Schwarz inequality,
\[
\left(
\frac1{n-1}\sum_{j=1}^{n-1}|1-aq_j|
\right)^2
\le
\frac1{n-1}\sum_{j=1}^{n-1}|1-aq_j|^2
=
\beta_s(1).
\]
On the other hand, since
$
1-aq_j=-\zeta_jq_j,
$
 the Cauchy--Schwarz inequality gives
\[
\begin{aligned}
\frac1{n-1}\sum_{j=1}^{n-1}|1-aq_j|
&=
\frac1{n-1}\sum_{j=1}^{n-1}|\zeta_jq_j|\\
&\le
\left(
\frac1{n-1}\sum_{j=1}^{n-1}|\zeta_j|^2
\right)^{1/2}
\left(
\frac1{n-1}\sum_{j=1}^{n-1}|q_j|^2
\right)^{1/2}\\
&=
\sqrt{V_*s}\\
&\le
\sqrt{Vs}.
\end{aligned}
\]
Therefore
\begin{equation}\label{eq:endpoint-first-moment}
\frac1{n-1}\sum_{j=1}^{n-1}|1-aq_j|
\le
\sqrt{\min\{\beta_s(1),Vs\}}.
\end{equation}

\step{Step 2. Choose the endpoint bound \(\kappa\).}
Choose \(\kappa\in[0,1]\) such that
\begin{equation}\label{eq:kappa}
\kappa^2\ge\min\{\beta_s(1),Vs\}.
\end{equation}
Then \eqref{eq:endpoint-first-moment} implies
\[
\frac1{n-1}\sum_{j=1}^{n-1}|1-aq_j|
\le\kappa.
\]
Define
$
\ell(t)=1-(1-\kappa)t.
$
Then \(\ell(0)=1\) and \(\ell(1)=\kappa\).

\step{Step 3. Interpolate between \(t=0\) and \(t=1\).}
Since
\[
1-atq_j=(1-t)+t(1-aq_j),
\]
the triangle inequality gives
$
|1-atq_j|
\le
(1-t)+t|1-aq_j|.
$
Averaging over \(j\), we obtain
\[
\begin{aligned}
\frac1{n-1}\sum_{j=1}^{n-1}|1-atq_j|
&\le
(1-t)+
t\frac1{n-1}\sum_{j=1}^{n-1}|1-aq_j|\\
&\le
(1-t)+t\kappa\\
&=
\ell(t),
\end{aligned}
\]
which proves \eqref{eq:linear-majorant-aim}.

In view of \eqref{eq:quadratic-majorant} and \eqref{eq:linear-majorant-aim}, \(b(t)\) controls the averaged second moment of the factors
\(1-atq_j\), while \(\ell(t)\) controls their averaged moduli.

\subsection{Products with one factor omitted}
In Lemma~\ref{lem:direct}, the proof combines two estimates that originate in the arguments of
Tao and Zhang.  The bound with coefficient \(1\) is based on the
Maclaurin's inequality used in  
\cite[Proposition~11]{Tao26}, while the second-moment bound refines the omitted-factor
estimate underlying \cite[Lemma~3.2]{Zha26b} by retaining the factor
\(s\).  We shall also use a first-moment analogue and an interval
refinement in which the omitted factor is retained quantitatively.

\begin{lem}\label{lem:direct}
Let
$\mathcal A(t)=\sum_j|q_j|^2\prod_{k\ne j}|1-atq_k|$.
Then
\begin{equation}\label{eq:direct-bounds}
 \mathcal A(t)\le(n-1)\min\left\{
 \min\{1,c_n s\}\,b(t)^{(n-2)/2},\,
 \min\{1,\widetilde c_n s\}\,\ell(t)^{n-2}
 \right\}.
\end{equation}
If $I=[l,u]\subset[0,1]$, $a\le a_+$, and $\delta_I=\max\{0,1-a_+u\}$, then, for all $t\in I$,
\begin{equation}\label{eq:retained}
 \mathcal A(t)\le(n-1)s\min\left\{
 \left(\frac{(n-1)b(t)-\delta_I^2}{n-2}\right)_{+}^{(n-2)/2},\,
 \left(\frac{(n-1)\ell(t)-\delta_I}{n-2}\right)_{+}^{n-2}
 \right\}.
\end{equation}
\end{lem}

\begin{proof}
We first prove \eqref{eq:direct-bounds}.  Maclaurin's inequality gives
\begin{equation*}
\sum_{j=1}^{n-1}\prod_{k\ne j}|1-atq_k|
\le
(n-1)
\left(
\frac1{n-1}\sum_{j=1}^{n-1}|1-atq_j|
\right)^{n-2}.
\end{equation*}
Since \(|q_j|^2\le1\), it follows from
\eqref{eq:quadratic-majorant}, \eqref{eq:linear-majorant-aim}, and the
Cauchy--Schwarz inequality that
\begin{equation}\label{eq:direct-unit-bound}
\begin{aligned}
\mathcal A(t)
&\le
\sum_{j=1}^{n-1}\prod_{k\ne j}|1-atq_k|\\
&\le
(n-1)
\min\left\{
b(t)^{(n-2)/2},
\ell(t)^{n-2}
\right\}.
\end{aligned}
\end{equation}

We now retain the factor \(|q_j|^2\).  For each fixed \(j\), the
arithmetic--geometric mean inequality and \eqref{eq:quadratic-majorant}
give
\[
\begin{aligned}
\prod_{k\ne j}|1-atq_k|
&\le
\left(
\frac1{n-2}\sum_{k\ne j}|1-atq_k|^2
\right)^{(n-2)/2}\\
&\le
\left(
\frac{n-1}{n-2}b(t)
\right)^{(n-2)/2}\\
&=
c_n b(t)^{(n-2)/2}.
\end{aligned}
\]
Similarly, by \eqref{eq:linear-majorant-aim},
\[
\begin{aligned}
\prod_{k\ne j}|1-atq_k|
&\le
\left(
\frac1{n-2}\sum_{k\ne j}|1-atq_k|
\right)^{n-2}\\
&\le
\left(
\frac{n-1}{n-2}\ell(t)
\right)^{n-2}\\
&=
\widetilde c_n\ell(t)^{n-2}.
\end{aligned}
\]
Multiplying these two estimates by \(|q_j|^2\), summing over \(j\), and
using
$
\sum_{j=1}^{n-1}|q_j|^2=(n-1)s,
$
we obtain
\begin{equation}\label{eq:direct-second-moment-bound}
\mathcal A(t)
\le
(n-1)\min\left\{
c_ns\,b(t)^{(n-2)/2},
\widetilde c_ns\,\ell(t)^{n-2}
\right\}.
\end{equation}
Combining \eqref{eq:direct-unit-bound} and
\eqref{eq:direct-second-moment-bound} proves \eqref{eq:direct-bounds}.

It remains to prove \eqref{eq:retained}.  Suppose that \(t\in I=[l,u]\).
Since \(a\le a_+\), \(t\le u\), and \(|q_j|<1\), the reverse triangle
inequality gives
\[
 |1-atq_j|
 \ge (1-at|q_j|)_+
 \ge (1-a_+u)_+
 =\delta_I.
\]
Hence, for each fixed \(j\),
\[
\sum_{k\ne j}|1-atq_k|^2
\le
(n-1)b(t)-\delta_I^2,
\]
while
\[
\sum_{k\ne j}|1-atq_k|
\le
(n-1)\ell(t)-\delta_I.
\]
Applying the arithmetic--geometric mean inequality to the \(n-2\)
remaining factors therefore gives
\[
\prod_{k\ne j}|1-atq_k|
\le
\left(
\frac{(n-1)b(t)-\delta_I^2}{n-2}
\right)_{+}^{(n-2)/2}
\]
and
\[
\prod_{k\ne j}|1-atq_k|
\le
\left(
\frac{(n-1)\ell(t)-\delta_I}{n-2}
\right)_{+}^{n-2}.
\]
Finally, multiplying by \(|q_j|^2\), summing over \(j\), and again using
$
\sum_{j=1}^{n-1}|q_j|^2=(n-1)s,
$
yields \eqref{eq:retained}.
\end{proof}

The positive parts in \eqref{eq:retained} allow us to enlarge interval bounds without affecting their validity. On any interval that contains admissible data, the quantities inside the positive parts are already nonnegative. We therefore keep the lower bound \(\delta_I\) fixed throughout \(I\), rather than replacing it by a \(t\)-dependent quantity, which is important when convexity is later used to estimate the corresponding integrals.

\subsection{Centered interpolation}

In the following lemma, the first estimate refines the centered interpolation bound
\cite[Lemma~5.1]{Zha26b}, replacing the single second-moment majorant
there by the minimum of the quadratic and linear majorants introduced
above.  In the present setting, however, \(1-a\mu t\) may vanish.
We therefore also require a second, nonsingular estimate, obtained by
using the second derivative of the interpolation instead of dividing
by \(1-a\mu t\).

\begin{lem}\label{lem:interpolation}
For every $t\in[0,1]$ with $1-a\mu t\ne0$,
\begin{equation}\label{eq:interpolation-first}
 |F(t)-(1-a\mu t)^{n-1}|
 \le\frac{(n-1)a^2t^2v}{2|1-a\mu t|}
 \min\{c_n b(t)^{(n-2)/2},\widetilde c_n\ell(t)^{n-2}\}.
\end{equation}
For every $t\in[0,1]$, including zeros of $1-a\mu t$,
\begin{equation}\label{eq:interpolation-second}
 |F(t)-(1-a\mu t)^{n-1}|
 \le\frac{(n-1)(n-2)a^2t^2v}{2}
 \min\{d_n b(t)^{(n-3)/2},\widetilde d_n\ell(t)^{n-3}\}.
\end{equation}
\end{lem}

\begin{proof}
\step{Step 1. The estimate with denominator.}
For $0\le\theta\le1$, define
$
G_\theta(t)
=
\prod_{j=1}^{n-1}
(1-a\mu t-at\theta w_j).
$
Since $q_j=\mu+w_j$, we have
\[
G_0(t)=(1-a\mu t)^{n-1}\quad \text{ and }\quad 
G_1(t)=F(t).
\]
Moreover,
\[
1-a\mu t-at\theta w_j
=
(1-\theta)(1-a\mu t)+\theta(1-atq_j).
\]

Using $\sum_{j=1}^{n-1}w_j=0$, a direct calculation gives
\[
\begin{aligned}
\frac1{n-1}\sum_{j=1}^{n-1}
|1-a\mu t-at\theta w_j|^2
&=
|1-a\mu t|^2+a^2t^2\theta^2v\\
&\le
|1-a\mu t|^2+a^2t^2v\\
&=
\beta_s(t)\\
&\le b(t).
\end{aligned}
\]
For the first absolute moment, the triangle inequality and
\eqref{eq:linear-majorant-aim} give
\[
\begin{aligned}
\frac1{n-1}\sum_{j=1}^{n-1}
|1-a\mu t-at\theta w_j|
&\le
(1-\theta)|1-a\mu t|
+\frac{\theta}{n-1}
\sum_{j=1}^{n-1}|1-atq_j|\\
&\le
\ell(t).
\end{aligned}
\]

Differentiating $G_\theta(t)$ with respect to $\theta$ gives
\[
\partial_\theta G_\theta(t)
=
-at\sum_{j=1}^{n-1}
w_j
\prod_{k\ne j}
(1-a\mu t-at\theta w_k).
\]
Multiplying by $1-a\mu t$ and using
\[
1-a\mu t
=
(1-a\mu t-at\theta w_j)+at\theta w_j
\]
together with $\sum_jw_j=0$, we obtain
\begin{equation}\label{eq:Gtheta-first}
(1-a\mu t)\,\partial_\theta G_\theta(t)
=
-a^2t^2\theta
\sum_{j=1}^{n-1}w_j^2
\prod_{k\ne j}
(1-a\mu t-at\theta w_k).
\end{equation}

The same arithmetic--geometric mean argument as in
Lemma~\ref{lem:direct}, now applied to the factors
$|1-a\mu t-at\theta w_k|$, gives
\[
\prod_{k\ne j}
|1-a\mu t-at\theta w_k|
\le
c_n b(t)^{(n-2)/2}
\]
and
\[
\prod_{k\ne j}
|1-a\mu t-at\theta w_k|
\le
\widetilde c_n\ell(t)^{n-2}.
\]
Since
$
\sum_{j=1}^{n-1}|w_j|^2=(n-1)v,
$
it follows from \eqref{eq:Gtheta-first} that
\[
|\partial_\theta G_\theta(t)|
\le
\frac{(n-1)a^2t^2\theta v}{|1-a\mu t|}
\min\{
c_n b(t)^{(n-2)/2},
\widetilde c_n\ell(t)^{n-2}
\}.
\]
Finally,
\[
F(t)-(1-a\mu t)^{n-1}
=
G_1(t)-G_0(t)
=
\int_0^1\partial_\theta G_\theta(t)\dd\theta.
\]
Integrating the preceding estimate and using
$
\int_0^1\theta\dd\theta=1/2
$
proves \eqref{eq:interpolation-first}.

\step{Step 2. The estimate without denominator.}
The argument in Step 1 is not useful when $1-a\mu t=0$.  To avoid
division by this quantity, we instead use the second derivative of the
interpolation.  Since
$
\partial_\theta G_0(t)=0
$
by $\sum_jw_j=0$, Taylor's formula with integral remainder gives
\begin{equation}\label{eq:Gtheta-Taylor}
G_1(t)-G_0(t)
=
\int_0^1(1-\theta)\,
\partial_\theta^2G_\theta(t)\dd\theta.
\end{equation}
Differentiating once more, we obtain
\begin{equation}\label{eq:Gtheta-second}
\partial_\theta^2G_\theta(t)
=
a^2t^2
\sum_{\substack{j,k=1\\ j\ne k}}^{n-1}
w_jw_k
\prod_{\substack{r=1\\ r\ne j,k}}^{n-1}
(1-a\mu t-at\theta w_r).
\end{equation}

For the ordered pairs in \eqref{eq:Gtheta-second}, the Cauchy--Schwarz inequality gives
\[
\begin{aligned}
\sum_{j\ne k}|w_jw_k|
&=
\left(\sum_{j=1}^{n-1}|w_j|\right)^2
-\sum_{j=1}^{n-1}|w_j|^2\\
&\le
(n-2)\sum_{j=1}^{n-1}|w_j|^2\\
&=
(n-1)(n-2)v.
\end{aligned}
\]
When two factors are omitted, the arithmetic--geometric mean inequality,
together with the bounds established in Step~1, gives
\[
\prod_{\substack{r=1\\ r\ne j,k}}^{n-1}
|1-a\mu t-at\theta w_r|
\le
d_n b(t)^{(n-3)/2}
\]
and
\[
\prod_{\substack{r=1\\ r\ne j,k}}^{n-1}
|1-a\mu t-at\theta w_r|
\le
\widetilde d_n\ell(t)^{n-3}.
\]
Consequently,
\[
|\partial_\theta^2G_\theta(t)|
\le
(n-1)(n-2)a^2t^2v
\min\{
d_n b(t)^{(n-3)/2},
\widetilde d_n\ell(t)^{n-3}
\}.
\]
Substituting this estimate into \eqref{eq:Gtheta-Taylor} and using
$
\int_0^1(1-\theta)\dd\theta=1/2
$
gives
\[
|G_1(t)-G_0(t)|
\le
\frac{(n-1)(n-2)a^2t^2v}{2}
\min\{
d_n b(t)^{(n-3)/2},
\widetilde d_n\ell(t)^{n-3}
\}.
\]
Since
$
G_1(t)=F(t)
$
and
$
G_0(t)=(1-a\mu t)^{n-1},
$
this is precisely \eqref{eq:interpolation-second}.
\end{proof}

\begin{lem}\label{lem:integral-tests}
Suppose that \(\mathcal C\) is an upper bound for \(|a\mu J|\), and let
\(\mathcal U:[0,1]\to[0,\infty)\) satisfy
\[
|F(t)-(1-a\mu t)^{n-1}|
\le
a^2t^2v\,\mathcal U(t),
\qquad 0\le t\le1.
\]
Set
\begin{equation}\label{eq:K-definition}
K
:=
na^3|\mu|
\int_0^1t^2\mathcal U(t)\dd t.
\end{equation}
Then every normalized counterexample satisfies
\begin{equation}\label{eq:direct-test}
1
\le
\kappa^{n-1}
+\frac{n-1}{n}\mathcal C
+a^2\int_0^1t\,\mathcal A(t)\dd t,
\end{equation}
and
\begin{equation}\label{eq:absolute-test}
1
\le
\mathcal C+(Vv)^{n/2}+Kv.
\end{equation}
Moreover,
\begin{equation}\label{eq:ratio-test}
1
\le
(1+|\mu|)
\left\{
V^{n/2}(1-|\mu|^2)^{(n-2)/2}+K
\right\}.
\end{equation}
\end{lem}

\begin{proof}
We first prove \eqref{eq:direct-test}.  Taking absolute values in
\eqref{eq:differential} gives
\[
1
\le
|F(1)|
+(n-1)a|\mu|
\left|\int_0^1F(t)\dd t\right|
+|\mathcal R|.
\]
By \eqref{eq:linear-majorant-aim} at \(t=1\) and the
arithmetic--geometric mean inequality,
\begin{equation}\label{eq:F1-kappa}
\begin{aligned}
|F(1)|
&=
\prod_{j=1}^{n-1}|1-aq_j|\\
&\le
\left(
\frac1{n-1}\sum_{j=1}^{n-1}|1-aq_j|
\right)^{n-1}\\
&\le
\kappa^{n-1}.
\end{aligned}
\end{equation}
Also, \eqref{eq:origin} yields
\begin{equation}\label{eq:origin-C-bound}
(n-1)a|\mu|
\left|\int_0^1F(t)\dd t\right|
=
\frac{n-1}{n}|a\mu J|
\le
\frac{n-1}{n}\mathcal C.
\end{equation}
Finally, by the definition of \(\mathcal R\),
\begin{equation}\label{eq:remainder-A-bound}
\begin{aligned}
|\mathcal R|
&\le
a^2\int_0^1
t\sum_{j=1}^{n-1}|q_j|^2
\prod_{k\ne j}|1-atq_k|\dd t\\
&=
a^2\int_0^1t\,\mathcal A(t)\dd t.
\end{aligned}
\end{equation}
Combining \eqref{eq:F1-kappa}, \eqref{eq:origin-C-bound}, and
\eqref{eq:remainder-A-bound} proves \eqref{eq:direct-test}.

We next derive \eqref{eq:absolute-test}.  Taking absolute values in
\eqref{eq:center-identity}, we obtain
\[
\begin{aligned}
1
&\le
|a\mu J|
+|1-a\mu|^n
+na|\mu|\int_0^1
|F(t)-(1-a\mu t)^{n-1}|\dd t.
\end{aligned}
\]
By \eqref{eq:moment-b},
$
|1-a\mu|^2\le Vv,
$
and by the defining property of \(\mathcal U\),
\[
na|\mu|\int_0^1
|F(t)-(1-a\mu t)^{n-1}|\dd t
\le
Kv.
\]
Since \(|a\mu J|\le\mathcal C\), it follows that
\[
1
\le
\mathcal C+(Vv)^{n/2}+Kv,
\]
which is \eqref{eq:absolute-test}.

It remains to prove \eqref{eq:ratio-test}.  From \eqref{eq:core},
together with \(\Phi_{n-1}(a)\le1\) and \(|q_j|<1\), we have
\[
|a\mu J|
\le
|\mu|.
\]
Using this sharper bound directly in \eqref{eq:center-identity} gives
\begin{equation}\label{eq:ratio-preliminary}
1-|\mu|
\le
(Vv)^{n/2}+Kv.
\end{equation}
Since $
v=s-|\mu|^2\le1-|\mu|^2,
$
the right-hand side of \eqref{eq:ratio-preliminary} satisfies
\[
\begin{aligned}
(Vv)^{n/2}+Kv
&=
v\left\{
V^{n/2}v^{(n-2)/2}+K
\right\}\\
&\le
(1-|\mu|^2)
\left\{
V^{n/2}(1-|\mu|^2)^{(n-2)/2}+K
\right\}.
\end{aligned}
\]
Moreover, \(|\mu|<1\), since \(|q_j|<1\) for every \(j\).  Hence
\(1-|\mu|>0\), and dividing by \(1-|\mu|\) gives
\[
1
\le
(1+|\mu|)
\left\{
V^{n/2}(1-|\mu|^2)^{(n-2)/2}+K
\right\},
\]
which proves \eqref{eq:ratio-test}.
\end{proof}

\begin{rem}\label{rem:ratio-test}
The derivation of \eqref{eq:ratio-test} uses the actual bound
$
|a\mu J|\le|\mu|
$
before any interval enlargement is made.  An independently chosen
interval upper bound \(\mathcal C\) need not satisfy
\(\mathcal C\le|\mu|\), and therefore cannot in general be substituted
for \(|a\mu J|\) in this argument.
\end{rem}
\subsection{Convex integration on a rational mesh}

Let \(f\) be a nonnegative convex function on \([l,u]\).  By convexity,
\(f\) lies below the affine chord joining its endpoint values; that is,
\[
f(t)
\le
\frac{u-t}{u-l}f(l)
+
\frac{t-l}{u-l}f(u),
\qquad l\le t\le u.
\]
Multiplying this inequality by \(t\) and \(t^2\), respectively, and
integrating over \([l,u]\), we obtain
\begin{align}
 \int_l^u tf(t)\dd t
 &\le\frac{(u-l)(2l+u)}6f(l)
      +\frac{(u-l)(l+2u)}6f(u),\label{eq:chord-one}\\
 \int_l^u t^2f(t)\dd t
 &\le\frac{(u-l)(3l^2+2lu+u^2)}{12}f(l)
      +\frac{(u-l)(l^2+2lu+3u^2)}{12}f(u).\label{eq:chord-two}
\end{align}
All four coefficients on the right-hand sides are nonnegative for
\(0\le l\le u\le1\).

We now explain how these bounds are applied to the majorants introduced
above.  Since \(b\) is nonnegative and convex, \(b^\alpha\) is convex
whenever \(\alpha\ge1\).  Likewise, \(\ell\) is nonnegative and affine,
so \(\ell^\alpha\) is convex for every \(\alpha\ge1\).  Consequently,
each of the individual majorants occurring in
\eqref{eq:direct-bounds} is convex for \(n\ge4\).  The same conclusion
holds for the retained bounds in \eqref{eq:retained}: taking a positive
part preserves convexity, and composing a nonnegative convex function
with \(x\mapsto x^\alpha\), \(\alpha\ge1\), again gives a convex
function.

For the centered interpolation estimates, the corresponding powers are
also convex when \(n\ge5\).  In particular, the exponents
\[
\frac{n-2}{2},\qquad n-2,\qquad
\frac{n-3}{2},\qquad n-3
\]
are then all at least one.  On a mesh interval \(I=[l,u]\) for which
\(|1-a\mu t|\) has a positive constant lower bound, we replace the
denominator in \eqref{eq:interpolation-first} by that constant lower
bound and apply \eqref{eq:chord-two} to each resulting convex
majorant.  If no positive lower bound for \(|1-a\mu t|\) is available
on \(I\), we use only the nonsingular estimate
\eqref{eq:interpolation-second}.

Whenever several convex majorants are available, we integrate each one
separately using \eqref{eq:chord-one} or \eqref{eq:chord-two}, and only
afterwards take the smallest of the resulting upper bounds.  This order
is important: although each individual majorant is convex, their
pointwise minimum need not be convex, so the chord inequalities cannot
in general be applied directly to that minimum.

For \(n=4\), the exponent \((n-3)/2=1/2\) occurring in
\eqref{eq:interpolation-second} is below one, and the preceding
convexity argument no longer applies to that branch.  We therefore do
not use the centered integral tests in degree \(4\).

\section{The analytic argument for large degrees}\label{sec:large}

The covariance and integral estimates developed above are already strong
enough to exclude all sufficiently large degrees, without any finite
verification.  We use the cutoff \(n=100001\); the estimates below leave
a uniform margin in each of the four ranges of \(a\).

\begin{prop}\label{prop:large}
There is no normalized counterexample with $n\ge100001$.
\end{prop}

\begin{proof}
Throughout the proof we use \(V=1\).  From \eqref{eq:moment-a},
$
|\mu-a|^2\le a^2(1-s),
$
and hence
\[
x\ge a-|a-\mu|
\ge a(1-\sqrt{1-s}).
\]
On the other hand, \eqref{eq:moment-b} gives
$
s(1+a^2)\ge1,
$
so that
\[
1-s\le\frac{a^2}{1+a^2}.
\]
Consequently,
\begin{equation}\label{eq:xlower}
 x\ge a(1-\sqrt{1-s})
 \ge a\left(1-\frac a{\sqrt{1+a^2}}\right)
 \ge\frac{a}{2(1+a^2)}.
\end{equation}
For the last inequality, rationalization gives
\[
1-\frac a{\sqrt{1+a^2}}
=
\frac{1}
{\sqrt{1+a^2}\bigl(\sqrt{1+a^2}+a\bigr)}
\ge
\frac1{2(1+a^2)}.
\]

Set
$
g(a)=a^2/[4(1+a^2)^2].
$
Since \(\rho\ge|\mu|^2\ge x^2\), \eqref{eq:xlower} and the definition
of \(\beta_s\) imply
\[
\beta_s(t)
=
1-\rho t-a^2s\,t(1-t)
\le
1-g(a)t.
\]

We first obtain a consequence that will be used repeatedly.  Suppose
more generally that
\[
\beta_s(t)\le1-gt,
\qquad 0\le t\le1,
\]
for some \(g>0\).  By the arithmetic--geometric mean inequality,
\[
\begin{aligned}
|F(1)|
&=
\prod_{j=1}^{n-1}|1-aq_j|\\
&\le
\left(
\frac1{n-1}\sum_{j=1}^{n-1}|1-aq_j|^2
\right)^{(n-1)/2}\\
&=
\beta_s(1)^{(n-1)/2}\\
&\le
(1-g)^{(n-1)/2}\\
&\le
e^{-(n-1)g/2}.
\end{aligned}
\]
Likewise, Lemma~\ref{lem:direct}, with \(b=\beta_s\), gives
\[
\mathcal A(t)
\le
(n-1)\beta_s(t)^{(n-2)/2}
\le
(n-1)e^{-(n-2)gt/2}.
\]
It follows from the definition of \(\mathcal R\) that
\[
\begin{aligned}
|\mathcal R|
&\le
a^2(n-1)
\int_0^1t\,e^{-(n-2)gt/2}\dd t\\
&\le
a^2(n-1)
\int_0^\infty t\,e^{-(n-2)gt/2}\dd t\\
&=
\frac{4(n-1)a^2}{(n-2)^2g^2}.
\end{aligned}
\]
Thus
\begin{equation}\label{eq:exp-direct}
 |F(1)|\le e^{-(n-1)g/2},\qquad
 |\mathcal R|\le\frac{4(n-1)a^2}{(n-2)^2g^2}.
\end{equation}

Taking absolute values in \eqref{eq:differential}, and then using
\eqref{eq:origin} and \eqref{eq:core}, gives
\[
\begin{aligned}
1
&\le
|F(1)|
+\frac{n-1}{n}|a\mu J|
+|\mathcal R|\\
&\le
|F(1)|+\psi(a)+|\mathcal R|.
\end{aligned}
\]
Hence every normalized counterexample must satisfy
\begin{equation}\label{eq:large-direct}
 1\le |F(1)|+\psi(a)+|\mathcal R|.
\end{equation}
We now derive a contradiction in four ranges of \(a\).

\step{Step 1. \(0<a\le1/10\).}
Since \(|a-\zeta_j|>1\) for every \(j\),
\[
|P'(a)|
=
n\prod_{j=1}^{n-1}|a-\zeta_j|
>n.
\]
On the other hand,
$
P'(a)=\prod_{j=1}^{n-1}(a-z_j),
$
and, as in the proof of Lemma~\ref{lem:basic-product},
\[
\frac1{n-1}\sum_{j=1}^{n-1}|a-z_j|^2
=
\frac{n(1+a^2)}{n-1}.
\]
The arithmetic--geometric mean inequality therefore yields
\[
n
<
\left(
\frac{n(1+a^2)}{n-1}
\right)^{(n-1)/2}.
\]
Taking logarithms gives
\[
(n-1)\log(1+a^2)
>
2\log n
-(n-1)\log\left(1+\frac1{n-1}\right).
\]
Since
$
(n-1)\log\left(1+\frac1{n-1}\right)<1
$
and \(\log(1+a^2)<a^2\), we obtain
\begin{equation}\label{eq:small-a-n}
 (n-1)a^2>2\log n-1>19.
\end{equation}

For \(a\le1/10\), the middle expression in \eqref{eq:xlower} satisfies
\[
a\left(1-\frac a{\sqrt{1+a^2}}\right)
\ge
\frac{9a}{10}.
\]
Hence
\[
\rho\ge x^2\ge\frac{81}{100}a^2>\frac45a^2,
\]
and therefore
\[
\beta_s(t)\le1-\frac45a^2t.
\]
Applying \eqref{eq:exp-direct} with \(g=\frac45a^2\), and using
\eqref{eq:small-a-n}, gives
\[
|F(1)|
\le
e^{-\frac25(n-1)a^2}
<
\frac1{100}.
\]
Moreover,
\[
|\mathcal R|
\le
\frac{25(n-1)}
{4(n-2)^2a^2}
<
\frac8{(n-1)a^2}
<
\frac8{19}.
\]
Finally, \(\psi\) is increasing on \((0,1)\), and hence
\[
\psi(a)\le\psi(1/10)<\frac16.
\]
Substitution into \eqref{eq:large-direct} gives
\[
1
<
\frac1{100}+\frac16+\frac8{19}
<1,
\]
a contradiction.

\step{Step 2. \(a\in[1/10,3/4]\cup[5/4,2]\).}
The function
$
g(a)=\frac{a^2}{4(1+a^2)^2}
$
is increasing on \((0,1]\) and decreasing on \([1,\infty)\).  Hence,
on the present range,
$
g(a)\ge\frac1{500}.
$
Equation~\eqref{eq:exp-direct} therefore gives
\[
|F(1)|
\le
e^{-(n-1)/1000}
<
\frac1{1000}.
\]

Set
$
H_0(a)=\frac{16(1+a^2)^4}{a^2}.
$
Since \(g=g(a)\), the second estimate in \eqref{eq:exp-direct} becomes
\[
|\mathcal R|
\le
\frac{4H_0(a)(n-1)}{(n-2)^2}
<
\frac{5H_0(a)}{n-1}.
\]
The required bounds for \(\psi\) and \(H_0\) are summarized in
Table~\ref{tab:large}.

\begin{table}[ht]
\centering
\caption{Bounds away from \(a=1\) in the large-degree argument.}
\label{tab:large}
\begin{tabular}{@{}ccc@{}}
\toprule
Range of \(a\) & Upper bound for \(\psi(a)\) & Upper bound for \(H_0(a)\)\\
\midrule
$[1/10,1/2]$ & $3/4$ & $1665$\\
$[1/2,3/4]$ & $19/20$ & $170$\\
$[5/4,3/2]$ & $19/20$ & $794$\\
$[3/2,2]$ & $17/20$ & $2500$\\
\bottomrule
\end{tabular}
\end{table}

To verify the table, recall that \(\psi\) increases on \((0,1]\) and
decreases on \([1,\infty)\).  If \(u=a^2\), then
\[
H_0(a)=16\frac{(1+u)^4}{u},
\]
whose logarithmic derivative has the sign of \(3u-1\).  Thus only the
appropriate endpoints of the four intervals need to be checked.  The
bounds for \(\psi\) follow, for example, from
\[
e^{3/8}<\frac32,\qquad
e^{7/32}<\frac{19}{15},\qquad
e^{9/32}>\frac{25}{19},\qquad
e^{5/8}>\frac{30}{17}.
\]
These elementary exponential comparisons may be verified directly from
the exponential series; an exact rational check is included in the
supplement.

Since \(n-1\ge100000\), the right-hand side of
\eqref{eq:large-direct} is, in each row of the table, strictly less
than one.  The largest of the four resulting upper bounds is
\[
\frac1{1000}
+\frac{19}{20}
+\frac{5\cdot794}{100000}
=
\frac{9907}{10000}
<1.
\]
This again contradicts \eqref{eq:large-direct}.

\step{Step 3. \(3/4\le a\le5/4\).}
From \eqref{eq:xlower},
\[
|\mu|\ge x\ge\frac6{25}.
\]
Set
$
B_*=\frac{589}{625}.
$
Since \(\rho\ge|\mu|^2\), we have
\[
\rho\ge\frac{36}{625},
\qquad
\beta_s(1)=1-\rho\le B_*.
\]
Also, \eqref{eq:moment-b} gives
\[
s(1+a^2)\ge1,
\]
and hence, since \(a\ge3/4\),
\[
a^2s
\ge
\frac{a^2}{1+a^2}
\ge
\frac9{25}.
\]

For \(0\le t\le1/2\),
\[
\begin{aligned}
\beta_s(t)
&=
1-\rho t-a^2s\,t(1-t)\\
&\le
1-t\left(
\frac{36}{625}+\frac9{25}(1-t)
\right)\\
&\le
1-\frac t5.
\end{aligned}
\]
Moreover,
\[
\beta_s(1/2)
\le
1-\frac{18}{625}-\frac9{100}
=
\frac{2203}{2500}
<
\frac9{10}.
\]
Since \(\beta_s\) is convex, its maximum on \([1/2,1]\) is attained at
an endpoint.  Together with \(\beta_s(1)\le B_*\), this gives
\begin{equation}\label{eq:central-beta}
 \beta_s(t)\le1-t/5\quad(0\le t\le1/2),\qquad
 \beta_s(t)\le B_*\quad(1/2\le t\le1).
\end{equation}
Finally, for \(0\le t\le1/2\),
\[
|1-a\mu t|
\ge
1-a|\mu|t
\ge
1-\frac58
=
\frac38.
\]

Let
\[
E
=
n|a\mu|
\int_0^1
|F(t)-(1-a\mu t)^{n-1}|\dd t,
\]
and write \(E=E_1+E_2\), where the two terms correspond to
\([0,1/2]\) and \([1/2,1]\), respectively.

For \(E_1\), use \eqref{eq:interpolation-first},
\eqref{eq:central-beta}, \(c_n<5/3\),
\(a\le5/4\), \(|a\mu|\le5/4\), and
\(|1-a\mu t|\ge3/8\).  Since
\[
\left(1-\frac t5\right)^{(n-2)/2}
\le
e^{-(n-2)t/10},
\]
we obtain
\begin{align}
E_1
&\le
\frac{625}{144}n(n-1)v
\int_0^\infty
t^2e^{-(n-2)t/10}\dd t \notag\\
&=
\frac{78125}{9}
\frac{n(n-1)}{(n-2)^3}v
\le
\frac{10000}{n-1}v
\le
\frac v{10}.
\label{eq:Efirst}
\end{align}

For \(E_2\), we use the nonsingular estimate
\eqref{eq:interpolation-second}, together with \(d_n<3\) and the
second estimate in \eqref{eq:central-beta}.  Since
$
\int_{1/2}^1t^2\dd t=\frac7{24},
$
we obtain
\begin{align}
E_2
&\le
2n(n-1)^2vB_*^{(n-3)/2}\notag\\
&\le
2n(n-1)^2v\,e^{-(n-1)/40}\notag\\
&\le
\frac{6\cdot10^{12}}{(n-1)^3}v
<
\frac v{10},
\label{eq:Esecond}
\end{align}
where for the second inequality, note that
\[
-\log B_*
=
-\log\left(1-\frac{36}{625}\right)
>
\frac{36}{625}
>
\frac1{18},
\]
and hence, for \(n\ge100001\),
\[
B_*^{(n-3)/2}
\le
e^{-(n-1)/40}.
\]
For the third inequality, we use
\[
e^{(n-1)/40}
\ge
\frac{((n-1)/40)^6}{720}.
\]

It remains to bound the endpoint term in
\eqref{eq:center-identity}.  Since
$
v\le1-|\mu|^2\le B_*,
$ 
we have
\[
\begin{aligned}
v^{n/2}
&=
v\,v^{(n-2)/2}\\
&\le
(1-|\mu|^2)B_*^{(n-2)/2}\\
&\le
2(1-|\mu|)B_*^{(n-2)/2}\\
&\le
\frac{1-|\mu|}{5}.
\end{aligned}
\]
Now \eqref{eq:center-identity}, \eqref{eq:moment-b},
\eqref{eq:core}, \eqref{eq:Efirst}, and \eqref{eq:Esecond} give
\[
\begin{aligned}
1
&\le
|\mu|
+v^{n/2}
+E_1+E_2\\
&\le
|\mu|
+\frac{1-|\mu|}{5}
+\frac v5\\
&\le
|\mu|
+\frac35(1-|\mu|)\\
&<1,
\end{aligned}
\]
where in the last two lines we used
\(v\le1-|\mu|^2\le2(1-|\mu|)\) and \(|\mu|<1\).
This is the desired contradiction.

\step{Step 4. \(a\ge2\).}
From \eqref{eq:moment-a}, with \(V=1\),
\[
|\mu-a|^2\le a^2(1-s).
\]
Expanding the left-hand side gives
\[
a^2s
\le
2ax-|\mu|^2
\le
2a|\mu|-|\mu|^2
\le
2a-1.
\]
Since the function \((2a-1)/a^2\) is decreasing for \(a\ge2\), it follows
that
\[
s\le\frac{2a-1}{a^2}\le\frac34.
\]

We also have
\[
\beta_s(1)
=
\frac1{n-1}\sum_{j=1}^{n-1}|1-aq_j|^2
\le1,
\]
because \(\beta_s(1)=1-\rho\) and \(\rho\ge0\).  Minkowski's inequality
therefore gives
\[
\begin{aligned}
a\sqrt{s}
&=
\left(
\frac1{n-1}\sum_{j=1}^{n-1}|aq_j|^2
\right)^{1/2}\\
&\le
1+
\left(
\frac1{n-1}\sum_{j=1}^{n-1}|1-aq_j|^2
\right)^{1/2}\\
&=
1+\sqrt{\beta_s(1)}
\le2.
\end{aligned}
\]
Thus
$
a^2s\le4.
$

Moreover, since \(1-aq_j=-\zeta_jq_j\), the Cauchy--Schwarz inequality gives
\[
\begin{aligned}
\frac1{n-1}\sum_{j=1}^{n-1}|1-aq_j|
&=
\frac1{n-1}\sum_{j=1}^{n-1}|\zeta_jq_j|\\
&\le
\sqrt{V_*s}\\
&\le
\sqrt{s}
\le
\frac{\sqrt3}{2}
<
\frac78.
\end{aligned}
\]
We may therefore take
\[
\ell(t)=1-\frac t8.
\]

Using the linear branch of Lemma~\ref{lem:direct},
\(\widetilde c_n<3\), and \(a^2s\le4\), we obtain
\[
\begin{aligned}
|\mathcal R|
&\le
a^2(n-1)\widetilde c_ns
\int_0^1t\left(1-\frac t8\right)^{n-2}\dd t\\
&\le
12(n-1)
\int_0^8t\left(1-\frac t8\right)^{n-2}\dd t\\
&=
\frac{768}{n}.
\end{aligned}
\]
Similarly,
\[
|F(1)|
\le
\left(\frac78\right)^{n-1}.
\]
Finally, \(\psi\) decreases on \([1,\infty)\), so
\[
\psi(a)
\le
\psi(2)
=
2e^{-3/2}
<
\frac9{20}.
\]
Substituting these bounds into \eqref{eq:large-direct} gives
\[
1
\le
\left(\frac78\right)^{n-1}
+\frac9{20}
+\frac{768}{n}
<1,
\]
since \(n\ge100001\).  This final contradiction completes the proof.
\end{proof}

\section{Product bounds for the small degrees}\label{sec:products}

The direct estimates become less effective when both the reciprocal mean
and the relevant products are close to their extremal values.  To obtain
the additional information needed in the small-degree range, we shall
retain more of the covariance structure of the original root distances
and of the variation in the reciprocal moduli.  The notation introduced
in Section~\ref{sec:notation} remains in force.

\subsection{An elementary symmetric inequality}

Recall the elementary symmetric polynomials \(e_k\) introduced in
Section~\ref{sec:notation}.  The following inequality provides the
refinement of the arithmetic--geometric mean estimate that will be used
in Lemma~\ref{lem:root-product}.

\begin{lem}\label{lem:symmetric}
Let \(n\ge4\), and let
\(\boldsymbol{\xi}=(\xi_1,\ldots,\xi_{n-1})\) satisfy $\xi_j\ge0$ for all $ 1\le j\le n-1$ and $
\sum_{j=1}^{n-1}\xi_j=n-1.$
Then
\begin{equation}\label{eq:symmetric}
e_{n-2}(\boldsymbol{\xi})
\le
3+(n-4)e_{n-1}(\boldsymbol{\xi}).
\end{equation}
\end{lem}

\begin{proof}
Consider
$
f(\boldsymbol{\xi})
=
e_{n-2}(\boldsymbol{\xi})
-(n-4)e_{n-1}(\boldsymbol{\xi})
$
on the compact simplex
\[
\left\{
\boldsymbol{\xi}\in[0,\infty)^{n-1}:
\sum_{j=1}^{n-1}\xi_j=n-1
\right\}.
\]
It is enough to show that the maximum of \(f\) is \(3\).

\step{Step 1. Boundary points.}
Suppose first that \(\boldsymbol{\xi}\) lies on the boundary of the
simplex.  If at least two coordinates vanish, then every product
appearing in \(e_{n-2}(\boldsymbol{\xi})\) contains a zero, and hence
\[
e_{n-2}(\boldsymbol{\xi})
=
e_{n-1}(\boldsymbol{\xi})
=
0.
\]
Thus \(f(\boldsymbol{\xi})=0\).

If exactly one coordinate vanishes, say \(\xi_{n-1}=0\), then
\[
e_{n-1}(\boldsymbol{\xi})=0,
\qquad
e_{n-2}(\boldsymbol{\xi})
=
\prod_{j=1}^{n-2}\xi_j.
\]
Since
$
\sum_{j=1}^{n-2}\xi_j=n-1,
$
the arithmetic--geometric mean inequality gives
\[
f(\boldsymbol{\xi})
=
\prod_{j=1}^{n-2}\xi_j
\le
\left(\frac{n-1}{n-2}\right)^{n-2}
<3.
\]
On the other hand, at
$
\boldsymbol{\xi}=(1,\ldots,1)
$
we have
\[
e_{n-2}(\boldsymbol{\xi})=n-1,
\qquad
e_{n-1}(\boldsymbol{\xi})=1,
\]
and therefore
\[
f(1,\ldots,1)=3.
\]
It remains only to exclude a nonconstant interior maximizer with value
larger than \(3\).

\step{Step 2. Interior stationary points.}
Let \(\boldsymbol{\xi}\) be an interior maximizer.  Since the constraint
is
$
\sum_{j=1}^{n-1}\xi_j=n-1,
$
the Lagrange multiplier equations imply
\[
\frac{\partial f}{\partial \xi_i}
=
\frac{\partial f}{\partial \xi_j}
\qquad
(1\le i,j\le n-1).
\]
For \(i\ne j\), a direct calculation gives
\begin{equation}\label{eq:symmetric-gradient}
\frac{\partial f}{\partial \xi_i}
-
\frac{\partial f}{\partial \xi_j}
=
(\xi_j-\xi_i)
\prod_{\substack{1\le \ell\le n-1\\ \ell\ne i,j}}
\xi_\ell
\left(
\sum_{\substack{1\le \ell\le n-1\\ \ell\ne i,j}}
\frac1{\xi_\ell}
-(n-4)
\right).
\end{equation}
Since all coordinates are positive, whenever \(\xi_i\ne\xi_j\),
\eqref{eq:symmetric-gradient} yields
\begin{equation}\label{eq:stationary}
\sum_{\substack{1\le \ell\le n-1\\ \ell\ne i,j}}
\frac1{\xi_\ell}
=
n-4.
\end{equation}

We claim that at most two distinct coordinate values can occur.
Indeed, suppose that \(\xi_i,\xi_j,\xi_k\) are three distinct values.
Applying \eqref{eq:stationary} to the pairs \((i,j)\) and \((i,k)\)
and subtracting the resulting identities gives
\[
\frac1{\xi_k}
=
\frac1{\xi_j},
\]
and hence \(\xi_j=\xi_k\), a contradiction.  Thus every nonconstant
interior stationary point has exactly two distinct coordinate values,
say \(u\ne w\).

\step{Step 3. Exclusion of nonconstant interior maxima.}
We first show that \(u\) and \(w\) cannot both occur at least twice.
Suppose that
$
\xi_i=\xi_j=u.
$
Along the tangent direction in which \(\xi_i\) is replaced by
\(u+t\) and \(\xi_j\) by \(u-t\), while all other coordinates are held
fixed, the second derivative at \(t=0\) is
\begin{equation}\label{eq:symmetric-second-variation}
-2
\prod_{\substack{1\le \ell\le n-1\\ \ell\ne i,j}}
\xi_\ell
\left(
\sum_{\substack{1\le \ell\le n-1\\ \ell\ne i,j}}
\frac1{\xi_\ell}
-(n-4)
\right).
\end{equation}
At a local maximum this quantity must be nonpositive.  Since all
coordinates are positive, \eqref{eq:symmetric-second-variation}
therefore implies
\[
\sum_{\substack{1\le \ell\le n-1\\ \ell\ne i,j}}
\frac1{\xi_\ell}
\ge n-4.
\]

Now choose a mixed pair consisting of one coordinate equal to \(u\)
and one coordinate equal to \(w\).  By \eqref{eq:stationary},
\[
\sum_{\substack{1\le \ell\le n-1\\ \ell\ne i,j}}
\frac1{\xi_\ell}
=
n-4
\]
for such a pair.  Comparing this identity with the preceding inequality
for a pair \(u,u\) gives
\[
\frac1w-\frac1u\ge0,
\]
and hence \(u\ge w\).  If \(w\) also occurs at least twice, the same
argument applied to a pair \(w,w\) gives
\[
\frac1u-\frac1w\ge0,
\]
and hence \(w\ge u\).  This contradicts \(u\ne w\).

Thus one of the two values occurs exactly once.  Relabeling if
necessary, suppose that \(u\) occurs \(n-2\) times and \(w\) occurs
once.  Applying \eqref{eq:stationary} to a mixed pair \(u,w\) leaves
\(n-3\) coordinates, all equal to \(u\), and therefore
\begin{equation}\label{eq:symmetric-two-value}
\frac{n-3}{u}=n-4.
\end{equation}
For \(n=4\), \eqref{eq:symmetric-two-value} is impossible.  For
\(n>4\), it gives
\[
u=\frac{n-3}{n-4}.
\]
Using
$
(n-2)u+w=n-1,
$
we then obtain
\[
w
=
n-1-\frac{(n-2)(n-3)}{n-4}
=
-\frac2{n-4}<0,
\]
contrary to the positivity of the coordinates.

Hence there is no nonconstant interior maximum.  The only remaining
interior maximizer is
\(\boldsymbol{\xi}=(1,\ldots,1)\), where \(f=3\).  Together with
Step~1, this proves \eqref{eq:symmetric}.
\end{proof}

\subsection{The original root distances}

Define
\begin{equation}\label{eq:ABC}
 \begin{gathered}
 A_0=\frac1{n-1}\sum_{j=1}^{n-1}(a-z_j)=\frac{na}{n-1},\qquad
 B_0=\frac1{n-1}\sum_{j=1}^{n-1}|a-z_j|^2=\frac{n(1+a^2)}{n-1},\\
 C_0=B_0-A_0^2=\frac{n(n-1-a^2)}{(n-1)^2}.
 \end{gathered}
\end{equation}
The logarithmic derivative identity of Meir and Sharma \cite[equation~(3.2)]{MS69} relates these root distances to the reciprocal critical-point mean. Tao \cite[Lemma~15]{Tao26} states its simple-zero case explicitly.

\begin{lem}[{\cite[equation~(3.2)]{MS69}}]\label{lem:inverse-root}
For the original root distances $a-z_j$ in \eqref{eq:ABC}, one has
\begin{equation*}
 \frac1{n-1}\sum_{j=1}^{n-1}\frac1{a-z_j}=\frac\mu2.
\end{equation*}
\end{lem}

\begin{lem}\label{lem:root-product}
If \(C_0>0\), then
\begin{equation}\label{eq:root-product}
 \prod_{j=1}^{n-1}|q_j|^2
 \ge
 \frac{n^2}{B_0^{n-1}}(1+\gamma),
 \qquad \text{ where }
 \gamma
:=
 \frac{n-1}{3}
 \frac{|A_0-B_0\mu/2|^2}{C_0}.
\end{equation}
If \(C_0=0\), then \(A_0-B_0\mu/2=0\), and the same bound holds with
\(\gamma=0\).
\end{lem}

\begin{proof}
Assume first that \(C_0>0\).  Set
\[
T
=
\frac1{n-1}\sum_{j=1}^{n-1}
\left|\frac1{a-z_j}\right|^2,
\qquad
M
=
\frac1{n-1}\sum_{j=1}^{n-1}\frac1{a-z_j}
=
\frac{\mu}{2},
\]
where the second identity follows from Lemma~\ref{lem:inverse-root}.

By the definitions of \(A_0,B_0,C_0\),
\[
\frac1{n-1}\sum_{j=1}^{n-1}
|(a-z_j)-A_0|^2
=
C_0,
\]
while
\[
\frac1{n-1}\sum_{j=1}^{n-1}
\left|
\frac1{a-z_j}-M
\right|^2
=
T-|M|^2.
\]
Moreover,
\[
\begin{aligned}
\frac1{n-1}\sum_{j=1}^{n-1}
\bigl((a-z_j)-A_0\bigr)
\left(\frac1{a-z_j}-M\right)
&=
1-A_0M.
\end{aligned}
\]
Hence the Cauchy--Schwarz inequality gives
\begin{equation}\label{eq:root-cauchy}
|1-A_0M|^2
\le
C_0(T-|M|^2).
\end{equation}

Expanding \eqref{eq:root-cauchy}, and using
\(C_0=B_0-A_0^2\), we obtain
\[
C_0T
\ge
1-2A_0\Re M+B_0|M|^2.
\]
Multiplying by \(B_0\) and observing that
\[
C_0+|A_0-B_0M|^2
=
B_0\bigl(1-2A_0\Re M+B_0|M|^2\bigr),
\]
we obtain
\begin{equation}\label{eq:root-covariance}
B_0T
\ge
1+\frac{|A_0-B_0M|^2}{C_0}.
\end{equation}

We now apply Lemma~\ref{lem:symmetric} to
\[
\xi_j
=
\frac{|a-z_j|^2}{B_0},
\qquad 1\le j\le n-1.
\]
By the definition of \(B_0\),
\[
\sum_{j=1}^{n-1}\xi_j=n-1.
\]
Since every \(\xi_j>0\),
\[
\begin{aligned}
e_{n-2}(\boldsymbol{\xi})
&=
e_{n-1}(\boldsymbol{\xi})
\sum_{j=1}^{n-1}\frac1{\xi_j}\\
&=
e_{n-1}(\boldsymbol{\xi})
B_0\sum_{j=1}^{n-1}
\frac1{|a-z_j|^2}\\
&=
e_{n-1}(\boldsymbol{\xi})(n-1)B_0T.
\end{aligned}
\]
Therefore \eqref{eq:symmetric} gives
\[
e_{n-1}(\boldsymbol{\xi})
\bigl((n-1)B_0T-(n-4)\bigr)
\le3.
\]
Using \eqref{eq:root-covariance}, we have
\[
\begin{aligned}
(n-1)B_0T-(n-4)
&\ge
(n-1)
\left(
1+\frac{|A_0-B_0M|^2}{C_0}
\right)
-(n-4)\\
&=
3+
(n-1)\frac{|A_0-B_0M|^2}{C_0}\\
&=
3(1+\gamma),
\end{aligned}
\]
because \(M=\mu/2\).  Consequently,
\begin{equation}\label{eq:symmetric-product-consequence}
e_{n-1}(\boldsymbol{\xi})(1+\gamma)\le1.
\end{equation}

Finally,
\[
e_{n-1}(\boldsymbol{\xi})
=
\frac{\prod_{j=1}^{n-1}|a-z_j|^2}{B_0^{n-1}}.
\]
Since
$
P'(a)
=
\prod_{j=1}^{n-1}(a-z_j)
=
n\prod_{j=1}^{n-1}(a-\zeta_j),
$
and \(q_j=(a-\zeta_j)^{-1}\), we obtain
\[
\prod_{j=1}^{n-1}|a-z_j|^2
=
|P'(a)|^2
=
\frac{n^2}{\prod_{j=1}^{n-1}|q_j|^2}.
\]
Thus
\[
e_{n-1}(\boldsymbol{\xi})
=
\frac{n^2}
{B_0^{n-1}\prod_{j=1}^{n-1}|q_j|^2}.
\]
Substituting this identity into
\eqref{eq:symmetric-product-consequence} gives \eqref{eq:root-product}.

It remains to consider \(C_0=0\).  In this case,
\[
\frac1{n-1}\sum_{j=1}^{n-1}
|(a-z_j)-A_0|^2
=
C_0
=
0,
\]
so
\[
a-z_j=A_0
\qquad
(1\le j\le n-1).
\]
Hence Lemma~\ref{lem:inverse-root} gives
\[
M=\frac1{A_0}=\frac{\mu}{2}.
\]
Since \(B_0=A_0^2\), it follows that
\[
A_0-B_0\frac{\mu}{2}
=
A_0-B_0M
=
0.
\]
Moreover,
\[
B_0^{n-1}
=
\prod_{j=1}^{n-1}|a-z_j|^2
=
\frac{n^2}{\prod_{j=1}^{n-1}|q_j|^2},
\]
so \eqref{eq:root-product} holds, in fact with equality, when
\(\gamma=0\).
\end{proof}

\subsection{Variation of the reciprocal moduli}

\begin{lem}\label{lem:radial-loss}
If \(V>0\), then
\begin{equation}\label{eq:radial-loss}
 \prod_{j=1}^{n-1}|q_j|^2
 \le
 s^{n-1}
 \exp\left(
 -\frac{(n-1)|\mu-as|^2}{2sV}
 \right).
\end{equation}
For \(V=0\), the bound
$
\prod_{j=1}^{n-1}|q_j|^2\le s^{n-1}
$
remains valid.
\end{lem}

\begin{proof}
We begin with the elementary inequality
\begin{equation}\label{eq:log-loss}
 \log u
 \le
 \log s+\frac{u-s}{s}-\frac{(u-s)^2}{2s},
 \qquad 0<u,s\le1.
\end{equation}
Indeed, let
\[
h(u)
=
\log u-\log s-\frac{u-s}{s}
+\frac{(u-s)^2}{2s}.
\]
Then \(h(s)=0\), while
\[
h'(u)
=
\frac{(u-1)(u-s)}{su}.
\]
If \(0<u<s\), then \(h'(u)\ge0\) on \((0,s)\), and hence
\(h(u)\le h(s)=0\).  If \(s<u\le1\), then \(h'(u)\le0\) on
\((s,1]\), and again \(h(u)\le h(s)=0\).  This proves
\eqref{eq:log-loss}.

We now apply \eqref{eq:log-loss} with \(u=|q_j|^2\) and sum over
\(1\le j\le n-1\).  Since
$
\sum_{j=1}^{n-1}(|q_j|^2-s)=0
$
by the definition of \(s\), the linear terms cancel, and we obtain
\begin{equation}\label{eq:radial-log-product}
\log\prod_{j=1}^{n-1}|q_j|^2
\le
(n-1)\log s
-\frac1{2s}
\sum_{j=1}^{n-1}(|q_j|^2-s)^2.
\end{equation}

It remains to estimate the last sum from below.  Since
\[
\mu-as
=
-\frac1{n-1}
\sum_{j=1}^{n-1}
\overline{\zeta_j}(|q_j|^2-s),
\]
the Cauchy--Schwarz inequality gives
\[
\begin{aligned}
|\mu-as|^2
&\le
\left(
\frac1{n-1}\sum_{j=1}^{n-1}|\zeta_j|^2
\right)
\left(
\frac1{n-1}\sum_{j=1}^{n-1}(|q_j|^2-s)^2
\right)\\
&=
V_*
\left(
\frac1{n-1}\sum_{j=1}^{n-1}(|q_j|^2-s)^2
\right)\\
&\le
V
\left(
\frac1{n-1}\sum_{j=1}^{n-1}(|q_j|^2-s)^2
\right).
\end{aligned}
\]
Thus, when \(V>0\),
\begin{equation}\label{eq:radial-variance-lower}
\sum_{j=1}^{n-1}(|q_j|^2-s)^2
\ge
\frac{(n-1)|\mu-as|^2}{V}.
\end{equation}
Substituting \eqref{eq:radial-variance-lower} into
\eqref{eq:radial-log-product} yields
\[
\log\prod_{j=1}^{n-1}|q_j|^2
\le
(n-1)\log s
-
\frac{(n-1)|\mu-as|^2}{2sV}.
\]
Exponentiating gives \eqref{eq:radial-loss}.

If \(V=0\), no division by \(V\) is needed.  The arithmetic--geometric
mean inequality and the definition of \(s\) give directly
\[
\prod_{j=1}^{n-1}|q_j|^2
\le
\left(
\frac1{n-1}\sum_{j=1}^{n-1}|q_j|^2
\right)^{n-1}
=
s^{n-1},
\]
as claimed.
\end{proof}

We next use the following product inequality from Tao's exposition. Its nonvanishing hypothesis is automatic for reciprocal critical-point coordinates.

\begin{lem}[{\cite[Lemma~13]{Tao26}}]\label{lem:tao-product}
If $0<|\xi_j|\le1$ for $1\le j\le N$, then
\[
 \left(\prod_{j=1}^N|\xi_j|\right)
 \sum_{j=1}^N\left|\frac1{\xi_j}-\overline{\xi_j}\right|
 \le1-\prod_{j=1}^N|\xi_j|^2.
\]
\end{lem}

This inequality and the reciprocal second moment give two bounds on $|a-\mu|$.

\begin{lem}\label{lem:radial-bounds}
If $(n-1)(1-s)<1$, then
\begin{equation}\label{eq:finite-radial}
 |a-\mu|^2\le\frac{(1-s)^2}{1-(n-1)(1-s)},\qquad
 V_*\le1+\frac{1-s}{1-(n-1)(1-s)}-a^2.
\end{equation}
Without this additional restriction,
\begin{equation}\label{eq:product-radial}
 (n-1)|a-\mu|
 \le
 \left(\prod_{j=1}^{n-1}|q_j|^2\right)^{-1/2}
 -
 \left(\prod_{j=1}^{n-1}|q_j|^2\right)^{1/2}.
\end{equation}
\end{lem}

\begin{proof}
Assume first that $(n-1)(1-s)<1$.  Since
$
\sum_{j=1}^{n-1}(1-|q_j|^2)
=
(n-1)(1-s)
$
and every summand is nonnegative, we have, for each \(j\),
\begin{equation}\label{eq:q-pointwise-lower}
|q_j|^2
\ge
1-(n-1)(1-s)
>0.
\end{equation}

Moreover, since
$
1/{q_j}=a-\zeta_j
$
and
$
\sum_{j=1}^{n-1}\zeta_j=0,
$
we have
\[
\frac1{n-1}\sum_{j=1}^{n-1}\frac1{q_j}=a.
\]
Therefore
\[
\begin{aligned}
|a-\mu|
&=
|a-\overline{\mu}|\\
&=
\left|
\frac1{n-1}\sum_{j=1}^{n-1}
\left(\frac1{q_j}-\overline{q_j}\right)
\right|\\
&\le
\frac1{n-1}\sum_{j=1}^{n-1}
\left|\frac1{q_j}-\overline{q_j}\right|\\
&=
\frac1{n-1}\sum_{j=1}^{n-1}
\frac{1-|q_j|^2}{|q_j|}.
\end{aligned}
\]
Using \eqref{eq:q-pointwise-lower}, we obtain
\[
\begin{aligned}
|a-\mu|
&\le
\frac{1}
{\sqrt{1-(n-1)(1-s)}}
\frac1{n-1}\sum_{j=1}^{n-1}(1-|q_j|^2)\\
&=
\frac{1-s}
{\sqrt{1-(n-1)(1-s)}}.
\end{aligned}
\]
Squaring gives the first inequality in \eqref{eq:finite-radial}.

To estimate \(V_*\), observe that
\[
\begin{aligned}
V_*
&=
\frac1{n-1}\sum_{j=1}^{n-1}|\zeta_j|^2\\
&=
\frac1{n-1}\sum_{j=1}^{n-1}
\left|a-\frac1{q_j}\right|^2\\
&=
\frac1{n-1}\sum_{j=1}^{n-1}|q_j|^{-2}-a^2,
\end{aligned}
\]
where we again used
$
(n-1)^{-1}\sum_j q_j^{-1}=a.
$
Furthermore,
\[
\frac1{|q_j|^2}
=
1+\frac{1-|q_j|^2}{|q_j|^2}
\le
1+
\frac{1-|q_j|^2}{1-(n-1)(1-s)}
\]
by \eqref{eq:q-pointwise-lower}.  Averaging over \(j\) therefore gives
\[
\frac1{n-1}\sum_{j=1}^{n-1}|q_j|^{-2}
\le
1+
\frac{1-s}{1-(n-1)(1-s)}.
\]
Consequently,
\[
V_*
\le
1+
\frac{1-s}{1-(n-1)(1-s)}
-a^2,
\]
which proves the second inequality in \eqref{eq:finite-radial}.

We now prove \eqref{eq:product-radial}, for which no assumption on
\((n-1)(1-s)\) is needed.  Applying Lemma~\ref{lem:tao-product} with
\(N=n-1\) and \(\xi_j=q_j\), we obtain
\begin{equation}\label{eq:tao-product-applied}
\left(\prod_{j=1}^{n-1}|q_j|\right)
\sum_{j=1}^{n-1}
\left|
\frac1{q_j}-\overline{q_j}
\right|
\le
1-\prod_{j=1}^{n-1}|q_j|^2.
\end{equation}
On the other hand,
\[
\begin{aligned}
\sum_{j=1}^{n-1}
\left(\frac1{q_j}-\overline{q_j}\right)
&=
(n-1)a-(n-1)\overline{\mu}\\
&=
(n-1)(a-\overline{\mu}),
\end{aligned}
\]
and hence, by the triangle inequality,
\begin{equation}\label{eq:radial-triangle}
(n-1)|a-\mu|
\le
\sum_{j=1}^{n-1}
\left|
\frac1{q_j}-\overline{q_j}
\right|.
\end{equation}
Combining \eqref{eq:tao-product-applied} and
\eqref{eq:radial-triangle} yields
\[
(n-1)|a-\mu|
\prod_{j=1}^{n-1}|q_j|
\le
1-\prod_{j=1}^{n-1}|q_j|^2.
\]
Since every \(q_j\ne0\), division by
\(\prod_{j=1}^{n-1}|q_j|\) gives
\[
\begin{aligned}
(n-1)|a-\mu|
&\le
\left(\prod_{j=1}^{n-1}|q_j|\right)^{-1}
-
\prod_{j=1}^{n-1}|q_j|\\
&=
\left(\prod_{j=1}^{n-1}|q_j|^2\right)^{-1/2}
-
\left(\prod_{j=1}^{n-1}|q_j|^2\right)^{1/2},
\end{aligned}
\]
which is \eqref{eq:product-radial}.
\end{proof}

Meir and Sharma \cite[Lemma~2]{MS69} used the same extremal principle to
maximize an inverse-power sum under pointwise and product constraints.
We shall need it here for two convex functions, so we include the short
argument.

\begin{lem}\label{lem:packing}
Suppose $
0<L\le\prod_{j=1}^{n-1}|q_j|^2,
s\le s_+,
$
and let \(0<\eta<s_+\le1\) satisfy
\begin{equation}\label{eq:eta}
 \eta\left(\frac{(n-1)s_+-\eta}{n-2}\right)^{n-2}<L.
\end{equation}
Then
\[
|q_j|^2>\eta,
\qquad 1\le j\le n-1.
\]
Assume in addition that \(L\le1\), and choose
\(k\in\{0,\ldots,n-2\}\) such that
$
\eta^{k+1}\le L\le\eta^k.
$
Set
$
z={L}/{\eta^k}.
$
Then
\begin{align}
 (n-1)|a-\mu|
 &\le
 M:=
 \frac{k(1-\eta)}{\sqrt\eta}
 +\frac{1-z}{\sqrt z},
 \label{eq:packing-distance}\\
 V_*
 &\le
 T-a^2,\qquad
 T:=
 \frac{k/\eta+1/z+n-k-2}{n-1}.
 \label{eq:packing-variance}
\end{align}
\end{lem}

\begin{proof}
We first prove the pointwise lower bound.  Suppose, to the contrary,
that
$
|q_j|^2=x\le\eta
$
for some \(j\).  Since
\[
\sum_{\ell=1}^{n-1}|q_\ell|^2
=(n-1)s
\le(n-1)s_+,
\]
the arithmetic--geometric mean inequality applied to the remaining
\(n-2\) terms gives
\begin{equation}\label{eq:packing-product-upper}
\prod_{\ell=1}^{n-1}|q_\ell|^2
\le
x\left(
\frac{(n-1)s_+-x}{n-2}
\right)^{n-2}.
\end{equation}
The function
\[
x\longmapsto
x\left(
\frac{(n-1)s_+-x}{n-2}
\right)^{n-2}
\]
is increasing on \(0\le x\le s_+\).  Indeed, apart from a positive
constant, its derivative has the sign of
$
(n-1)s_+-(n-1)x.
$
Hence \eqref{eq:packing-product-upper} and \(x\le\eta\) imply
\[
\prod_{\ell=1}^{n-1}|q_\ell|^2
\le
\eta\left(
\frac{(n-1)s_+-\eta}{n-2}
\right)^{n-2}
<L,
\]
contrary to the assumption
\(L\le\prod_{\ell=1}^{n-1}|q_\ell|^2\).
Thus
\begin{equation}\label{eq:packing-pointwise}
|q_j|^2>\eta
\qquad (1\le j\le n-1).
\end{equation}

Notice also that \(s_+>\eta\) implies
\[
\frac{(n-1)s_+-\eta}{n-2}>\eta.
\]
Therefore \eqref{eq:eta} gives
$
L>\eta^{n-1}.
$
Together with \(L\le1\), this guarantees the existence of
\(k\in\{0,\ldots,n-2\}\) satisfying
\[
\eta^{k+1}\le L\le\eta^k.
\]
In particular,
\[
\eta\le z=\frac{L}{\eta^k}\le1.
\]

We now turn to the extremal argument.  Set
$
t_j=-\log|q_j|^2,
$ and $
\lambda=-\log\eta.
$
By \eqref{eq:packing-pointwise},
\[
0\le t_j<\lambda.
\]
Moreover,
\[
\sum_{j=1}^{n-1}t_j
=
-\log\prod_{j=1}^{n-1}|q_j|^2
\le
-\log L.
\]
Since \(L>\eta^{n-1}\),
\[
-\log L<(n-1)\lambda.
\]

Let \(f\) be any convex nondecreasing function on \([0,\lambda]\).
Because \(f\) is nondecreasing, enlarging
\(\sum_jt_j\) up to \(-\log L\) cannot decrease
\(\sum_jf(t_j)\).  We may therefore impose
$
\sum_{j=1}^{n-1}t_j=-\log L.
$
If two coordinates \(t_i,t_j\) both lie strictly between \(0\) and
\(\lambda\), move an amount
$
\delta=\min\{t_i,\lambda-t_j\}
$
from the smaller coordinate to the larger one.  Thus
\[
(t_i,t_j)
\longmapsto
(t_i-\delta,t_j+\delta).
\]
The sum is unchanged, while convexity of \(f\) shows that
\[
f(t_i-\delta)+f(t_j+\delta)
\ge
f(t_i)+f(t_j).
\]
Repeating this operation leaves at most one coordinate strictly between
\(0\) and \(\lambda\).

Since
$
\eta^{k+1}\le L\le\eta^k,
$
taking logarithms gives
\[
k\lambda
\le
-\log L
\le
(k+1)\lambda.
\]
Furthermore,
\[
-\log L
=
k\lambda-\log z.
\]
Thus the extremal vector has
\[
k\ \text{coordinates equal to }\lambda,\qquad
1\ \text{coordinate equal to }-\log z,
\]
and the remaining \(n-k-2\) coordinates equal to \(0\).  Consequently,
for every convex nondecreasing \(f\),
\begin{equation}\label{eq:packing-convex}
\sum_{j=1}^{n-1}f(t_j)
\le
k f(\lambda)
+f(-\log z)
+(n-k-2)f(0).
\end{equation}

We first apply \eqref{eq:packing-convex} to
$
f(t)=2\sinh(t/2).
$
Since
$
1/{q_j}-\overline{q_j}
=
{(1-|q_j|^2)}/{q_j},
$
we have
\[
\left|
\frac1{q_j}-\overline{q_j}
\right|
=
\frac{1-|q_j|^2}{|q_j|}
=
2\sinh(t_j/2).
\]
Also,
\[
\sum_{j=1}^{n-1}
\left(
\frac1{q_j}-\overline{q_j}
\right)
=
(n-1)(a-\overline{\mu}).
\]
Hence
\begin{equation}\label{eq:packing-distance-preliminary}
(n-1)|a-\mu|
\le
\sum_{j=1}^{n-1}2\sinh(t_j/2).
\end{equation}
Now
\[
2\sinh(\lambda/2)
=
\frac{1-\eta}{\sqrt\eta},
\qquad
2\sinh\!\left(\frac{-\log z}{2}\right)
=
\frac{1-z}{\sqrt z}.
\]
Applying \eqref{eq:packing-convex} to
\eqref{eq:packing-distance-preliminary} therefore gives
\[
(n-1)|a-\mu|
\le
\frac{k(1-\eta)}{\sqrt\eta}
+\frac{1-z}{\sqrt z},
\]
which proves \eqref{eq:packing-distance}.

For the variance estimate, apply \eqref{eq:packing-convex} to
$
f(t)=e^t.
$
Since
$
e^{t_j}=|q_j|^{-2},
$
and
$
\frac1{n-1}\sum_{j=1}^{n-1}|q_j|^{-2}
=
a^2+V_*,
$
we obtain
\[
\begin{aligned}
(n-1)(a^2+V_*)
&=
\sum_{j=1}^{n-1}e^{t_j}\\
&\le
\frac{k}{\eta}
+\frac1z
+n-k-2.
\end{aligned}
\]
Dividing by \(n-1\) and subtracting \(a^2\) gives
\[
V_*
\le
\frac{k/\eta+1/z+n-k-2}{n-1}
-a^2
=
T-a^2,
\]
which is \eqref{eq:packing-variance}.
\end{proof}

\section{Centered coefficients and configurations near equality}\label{sec:coefficients}

The centered interpolation estimates control the full product in terms of
the variance \(v\).  For the small degrees, however, it is useful to retain
the individual coefficients of the centered product.  Write
\[
\bm w=(w_1,\ldots,w_{n-1})\in\C^{n-1}.
\]
Repeated application of Schoenberg's
inequality gives the coefficient estimate needed below.

\begin{lem}\label{lem:coefficients}
For \(2\le k\le n-1\),
\begin{equation}\label{eq:coefficients}
 |e_k(\bm w)|
 \le
 \binom{n-1}{k}
 \left(\frac{k-1}{n-2}v\right)^{k/2}.
\end{equation}
\end{lem}

\begin{proof}
Let
$
R(z)=\prod_{j=1}^{n-1}(z-w_j).
$
Since
$
\sum_{j=1}^{n-1}w_j=0,
$
the centroid of the zeros of \(R\) is zero.  The same is true for every
derivative of \(R\).

Fix \(2\le k\le n-2\), and let
\(\lambda_1,\ldots,\lambda_k\) be the zeros of
\(R^{(n-1-k)}\).  Applying Schoenberg's inequality \eqref{eq:schoenberg} successively to
$
R,R',\ldots,R^{(n-k-2)},
$
whose degrees are
$
n-1,n-2,\ldots,k+1,
$
gives
\[
\begin{aligned}
\sum_{j=1}^k|\lambda_j|^2
&\le
\left(
\prod_{m=k+1}^{n-1}\frac{m-2}{m}
\right)
\sum_{j=1}^{n-1}|w_j|^2\\
&=
\frac{k(k-1)}{(n-1)(n-2)}
\,(n-1)v.
\end{aligned}
\]
Thus
\begin{equation}\label{eq:coefficient-second-moment}
\sum_{j=1}^k|\lambda_j|^2
\le
\frac{k(k-1)}{n-2}v.
\end{equation}

Next, write
\[
R(z)
=
z^{n-1}
+\sum_{r=1}^{n-1}(-1)^r e_r(\bm w)z^{n-1-r}.
\]
The leading coefficient of \(R^{(n-1-k)}\) is
$
{(n-1)!}/{k!},
$
whereas its constant term is
$
(-1)^k(n-1-k)!e_k(\bm w).
$
Hence the constant term after making \(R^{(n-1-k)}\) monic is
\begin{equation}\label{eq:coefficient-constant-term}
(-1)^k\frac{e_k(\bm w)}{\binom{n-1}{k}}.
\end{equation}
Since the product of the zeros of this monic polynomial is
\((-1)^k\) times its constant term,
\eqref{eq:coefficient-constant-term} gives
\[
\frac{|e_k(\bm w)|}{\binom{n-1}{k}}
=
\prod_{j=1}^k|\lambda_j|.
\]
The arithmetic--geometric mean inequality, followed by
\eqref{eq:coefficient-second-moment}, now yields
\[
\begin{aligned}
\frac{|e_k(\bm w)|}{\binom{n-1}{k}}
&\le
\left(
\frac1k\sum_{j=1}^k|\lambda_j|^2
\right)^{k/2}\\
&\le
\left(
\frac{k-1}{n-2}v
\right)^{k/2}.
\end{aligned}
\]
This proves \eqref{eq:coefficients} for \(2\le k\le n-2\).

For \(k=n-1\), no differentiation is required.  Since
$
e_{n-1}(\bm w)=\prod_{j=1}^{n-1}w_j,
$
the arithmetic--geometric mean inequality gives
\[
|e_{n-1}(\bm w)|
\le
\left(
\frac1{n-1}\sum_{j=1}^{n-1}|w_j|^2
\right)^{(n-1)/2}
=
v^{(n-1)/2},
\]
which is precisely \eqref{eq:coefficients} in this case.
\end{proof}

For nonnegative integers \(k,p\) and \(d\ge0\), define
\begin{equation}\label{eq:integral-polynomial}
 S_{k,p}(d)
 =
 \sum_{j=0}^p\binom{k+j}{k}d^j,
 \qquad
 I_{k,p}(d)
 =
 \frac{S_{k,p}(d)}
 {(k+p+1)\binom{k+p}{k}}
 =
 \int_0^1t^k(1-t+dt)^p\dd t.
\end{equation}
To verify the integral identity, expand with respect to \(d\):
\[
(1-t+dt)^p
=
\sum_{j=0}^p
\binom pj d^jt^j(1-t)^{p-j}.
\]
Using the beta integral,
\[
\begin{aligned}
\int_0^1t^k(1-t+dt)^p\dd t
&=
\sum_{j=0}^p
\binom pj d^j
\int_0^1t^{k+j}(1-t)^{p-j}\dd t\\
&=
\frac1{(k+p+1)\binom{k+p}{k}}
\sum_{j=0}^p\binom{k+j}{k}d^j,
\end{aligned}
\]
which is the second identity in \eqref{eq:integral-polynomial}.

\begin{lem}\label{lem:coefficient-test}
If \(\mathcal C\ge|a\mu J|\) and \(d\ge|1-a\mu|\), then every normalized
counterexample satisfies
\begin{equation}\label{eq:coefficient-test}
 1\le\mathcal C+d^n
 +a|\mu|\sum_{k=2}^{n-1}
 \left(\frac{a^2v(k-1)}{n-2}\right)^{k/2}
 S_{k,n-1-k}(d).
\end{equation}
\end{lem}

\begin{proof}
Since \(q_j=\mu+w_j\) and \(e_1(\bm w)=0\), expansion of the product gives
\begin{equation}\label{eq:product-expansion}
 F(t)
 =
 (1-a\mu t)^{n-1}
 +
 \sum_{k=2}^{n-1}
 (-at)^k e_k(\bm w)(1-a\mu t)^{n-1-k}.
\end{equation}

We first establish a convenient bound for the remaining linear factor.
Since
\[
1-a\mu t
=
(1-t)+t(1-a\mu),
\]
the triangle inequality and the assumption
\(d\ge|1-a\mu|\) give
\begin{equation}\label{eq:center-linear-bound}
|1-a\mu t|
\le
1-t+dt,
\qquad 0\le t\le1.
\end{equation}

Taking absolute values in \eqref{eq:center-identity} and using
\(|a\mu J|\le\mathcal C\) and
\(|1-a\mu|\le d\), we obtain
\begin{equation}\label{eq:coefficient-center-bound}
1
\le
\mathcal C+d^n
+
na|\mu|
\int_0^1
|F(t)-(1-a\mu t)^{n-1}|\dd t.
\end{equation}

By \eqref{eq:product-expansion} and
\eqref{eq:center-linear-bound},
\[
\begin{aligned}
|F(t)-(1-a\mu t)^{n-1}|
&\le
\sum_{k=2}^{n-1}
a^kt^k|e_k(\bm w)|
(1-t+dt)^{n-1-k}.
\end{aligned}
\]
Using Lemma~\ref{lem:coefficients} and integrating term by term therefore
gives
\[
\begin{aligned}
&na|\mu|
\int_0^1
|F(t)-(1-a\mu t)^{n-1}|\dd t\\
&\qquad\le
na|\mu|
\sum_{k=2}^{n-1}
a^k
\binom{n-1}{k}
\left(\frac{k-1}{n-2}v\right)^{k/2}
I_{k,n-1-k}(d).
\end{aligned}
\]
Since \(k+(n-1-k)+1=n\), \eqref{eq:integral-polynomial} gives
\begin{equation}\label{eq:S-I-relation}
n\binom{n-1}{k}
I_{k,n-1-k}(d)
=
S_{k,n-1-k}(d).
\end{equation}
Consequently,
\[
\begin{aligned}
&na|\mu|
\int_0^1
|F(t)-(1-a\mu t)^{n-1}|\dd t\\
&\qquad\le
a|\mu|
\sum_{k=2}^{n-1}
\left(
\frac{a^2v(k-1)}{n-2}
\right)^{k/2}
S_{k,n-1-k}(d).
\end{aligned}
\]
Substituting this estimate into
\eqref{eq:coefficient-center-bound} proves
\eqref{eq:coefficient-test}.
\end{proof}
\subsection{Retaining the quadratic term}

The coefficient estimate in Lemma~\ref{lem:coefficient-test} treats all
centered coefficients by absolute value.  Near the limiting configuration,
however, the quadratic term carries a favorable sign.  We now retain this
information.

Assume that \(\mu\ne0\), and write
$
\mu=|\mu|e^{i\theta}.
$
Set
\[
q'_j=e^{-i\theta}q_j,
\qquad
w'_j=q'_j-|\mu|,
\qquad
\bm w'=(w'_1,\ldots,w'_{n-1}).
\]
Since
$
\frac1{n-1}\sum_{j=1}^{n-1}q_j=\mu,
$
we have
\[
\frac1{n-1}\sum_{j=1}^{n-1}q'_j=|\mu|,
\qquad
\sum_{j=1}^{n-1}w'_j=0.
\]
Moreover,
$
w'_j=e^{-i\theta}w_j,
$
so that
\[
\frac1{n-1}\sum_{j=1}^{n-1}|w'_j|^2=v.
\]

Define
\[
W
=
\frac1{n-1}
\sum_{j=1}^{n-1}
(\Rea q'_j-|\mu|)^2
=
\frac1{n-1}
\sum_{j=1}^{n-1}
(\Rea w'_j)^2.
\]
Since \(e_1(\bm w')=0\),
\[
e_2(\bm w')
=
-\frac12\sum_{j=1}^{n-1}(w'_j)^2.
\]
Writing \(w'_j=\Rea w'_j+i\Ima w'_j\), we therefore obtain
\[
\begin{aligned}
\Rea e_2(\bm w')
&=
-\frac12
\sum_{j=1}^{n-1}
\left((\Rea w'_j)^2-(\Ima w'_j)^2\right)\\
&=
\frac12\sum_{j=1}^{n-1}|w'_j|^2
-
\sum_{j=1}^{n-1}(\Rea w'_j)^2.
\end{aligned}
\]
Hence
\begin{equation}\label{eq:quadratic-real}
 \Rea e_2(\bm w')
 =
 \frac{n-1}{2}v-(n-1)W.
\end{equation}

To estimate \(W\), we shall use the following classical variance bound.

\begin{lem}[{\cite[p.~353]{BD00}}]\label{lem:bd}
Let \(\xi_1,\ldots,\xi_N\in[L,U]\), and let
\(\omega_1,\ldots,\omega_N\ge0\) satisfy
\(\sum_{j=1}^N\omega_j=1\).  If
$
b=\sum_{j=1}^N\omega_j\xi_j,
$
then
\[
\sum_{j=1}^N\omega_j(\xi_j-b)^2
\le
(U-b)(b-L).
\]
\end{lem}

For the reciprocal critical-point coordinates, the relevant interval can
be bounded explicitly.

\begin{lem}\label{lem:real-variance}
Suppose that \(\mu\ne0\) and
$
|q_j|^2\ge t_0\ge0,
 1\le j\le n-1.
$
Set
\begin{equation}\label{eq:real-lower}
 L_\mu
 =
 \max\left\{
 -1,\,
 (n-1)|\mu|-(n-2),\,
 \frac{(n-1)^2|\mu|^2+t_0-(n-2)^2}
 {2(n-1)|\mu|}
 \right\}.
\end{equation}
Then
\begin{equation}\label{eq:real-variance}
 \Rea q'_j\ge L_\mu,
 \qquad
 W
 \le
 \min\left\{
 (n-2)(1-|\mu|)^2,\,
 (1-|\mu|)(|\mu|-L_\mu)
 \right\}.
\end{equation}
\end{lem}

\begin{proof}
We first establish the lower bound for \(\Rea q'_j\).  Since
\(|q'_j|=|q_j|<1\),
$
\Rea q'_j\ge-1.
$
Also,
$
\sum_{j=1}^{n-1}q'_j=(n-1)|\mu|.
$
Since \(\Rea q'_k\le1\) for every \(k\),
\begin{equation}\label{eq:real-lower-second}
\Rea q'_j
=
(n-1)|\mu|
-
\sum_{k\ne j}\Rea q'_k\ge
(n-1)|\mu|-(n-2).
\end{equation}
Finally,
$
(n-1)|\mu|-q'_j
=
\sum_{k\ne j}q'_k,
$
and hence
\begin{equation*}
\left|(n-1)|\mu|-q'_j\right|
\le
\sum_{k\ne j}|q'_k|
\le
n-2.
\end{equation*}
Squaring gives
\[
(n-1)^2|\mu|^2
+|q'_j|^2
-2(n-1)|\mu|\Rea q'_j
\le
(n-2)^2.
\]
Since \(|q'_j|^2=|q_j|^2\ge t_0\), it follows that
\begin{equation}\label{eq:real-lower-third}
\Rea q'_j
\ge
\frac{(n-1)^2|\mu|^2+t_0-(n-2)^2}
{2(n-1)|\mu|}.
\end{equation}
Combining
$
\Rea q'_j\ge-1,
$
\eqref{eq:real-lower-second}, and
\eqref{eq:real-lower-third}
proves
\[
\Rea q'_j\ge L_\mu.
\]

We next bound \(W\).  Apply Lemma~\ref{lem:bd} with
\[
N=n-1,\qquad
Y_j=\Rea q'_j,\qquad
\omega_j=\frac1{n-1},
\]
and
\[
L=L_\mu,\qquad
U=1,\qquad
b=|\mu|.
\]
Since
$
\frac1{n-1}\sum_{j=1}^{n-1}\Rea q'_j=|\mu|,
$
we obtain
\begin{equation}\label{eq:W-bhatia-davis}
W
\le
(1-|\mu|)(|\mu|-L_\mu).
\end{equation}

For the second estimate, put
$
d_j=1-\Rea q'_j\ge0.
$
Then
\[
\sum_{j=1}^{n-1}d_j
=
(n-1)(1-|\mu|),
\]
and
\[
\begin{aligned}
W
&=
\frac1{n-1}\sum_{j=1}^{n-1}
\bigl(d_j-(1-|\mu|)\bigr)^2\\
&=
\frac1{n-1}\sum_{j=1}^{n-1}d_j^2
-(1-|\mu|)^2.
\end{aligned}
\]
Since \(d_j\ge0\),
\[
\sum_{j=1}^{n-1}d_j^2
\le
\left(\sum_{j=1}^{n-1}d_j\right)^2,
\]
and therefore
\begin{equation}\label{eq:W-simple}
W
\le
(n-2)(1-|\mu|)^2.
\end{equation}
Combining \eqref{eq:W-bhatia-davis} and \eqref{eq:W-simple} proves
\eqref{eq:real-variance}.
\end{proof}

Define
\begin{equation}\label{eq:K2}
 K_2(z)
 =
 nz\int_0^1t^2(1-zt)^{n-3}\dd t.
\end{equation}
The quadratic term in \eqref{eq:product-expansion} contributes
$
-a^2e_2(\bm w)K_2(a\mu)
$
to the last term of \eqref{eq:center-identity}.  Instead of estimating
this contribution by its absolute value, we shall retain its real part.

\begin{lem}\label{lem:phase}
Suppose
$
0<h_0:=a|\mu|\le1,
\varepsilon\ge|e^{i\theta}-1|,
$ and $
d\ge|1-a\mu|.
$
Then
\begin{equation}\label{eq:phase}
 |e^{2i\theta}K_2(a\mu)-K_2(h_0)|
 \le
 E_2,
\end{equation}
where
\begin{equation*}
 E_2
 :=
 nh_0\varepsilon
 \left\{
 3I_{2,n-3}(d)
 +(n-3)h_0I_{3,n-4}(d)
 \right\}.
\end{equation*}
If, in addition,
$
0<L_2\le\Rea e_2(\bm w'),
0\le K_-\le K_2(h_0),
$ and $
\mathcal C\ge|a\mu J|,
$
then
\begin{equation}\label{eq:signed-test}
 \begin{split}
 1\le{}&
 \mathcal C+d^n-a^2K_-L_2
 +\frac{(n-1)a^2v}{2}E_2\\
 &\quad+
 a|\mu|\sum_{k=3}^{n-1}
 \left(\frac{a^2v(k-1)}{n-2}\right)^{k/2}
 S_{k,n-1-k}(d).
 \end{split}
\end{equation}
\end{lem}

\begin{proof}
Since \(a\mu=h_0e^{i\theta}\), the definition \eqref{eq:K2} gives
\[
e^{2i\theta}K_2(a\mu)
=
nh_0e^{3i\theta}
\int_0^1
t^2(1-h_0e^{i\theta}t)^{n-3}\dd t.
\]
Moreover,
\[
K_2(h_0)
=
nh_0
\int_0^1
t^2(1-h_0t)^{n-3}\dd t.
\]
Since \(0<h_0\le1\),
\[
0\le1-h_0t
\le
|1-h_0e^{i\theta}t|
=
|1-a\mu t|.
\]
Also,
$
1-a\mu t=(1-t)+t(1-a\mu),
$
so the assumption \(d\ge|1-a\mu|\) gives
\begin{equation}\label{eq:phase-linear-bound}
|1-a\mu t|
\le
1-t+dt.
\end{equation}

Set
$
A=1-h_0e^{i\theta}t,
$ and $
B=1-h_0t.
$
Then
\begin{equation}\label{eq:phase-decomposition}
|e^{3i\theta}A^{n-3}-B^{n-3}|
\le
|e^{3i\theta}-1|\,|A|^{n-3}
+
|A^{n-3}-B^{n-3}|.
\end{equation}
For the first term on the right-hand side of
\eqref{eq:phase-decomposition}, we have
\[
|e^{3i\theta}-1|
\le
3|e^{i\theta}-1|
\le
3\varepsilon.
\]
Moreover, by \eqref{eq:phase-linear-bound},
\[
|A|=|1-a\mu t|
\le
1-t+dt.
\]
Hence
\begin{equation}\label{eq:phase-first-term}
|e^{3i\theta}-1|\,|A|^{n-3}
\le
3\varepsilon(1-t+dt)^{n-3}.
\end{equation}

For the second term on the right-hand side of
\eqref{eq:phase-decomposition}, telescoping the difference of powers gives
\begin{equation}\label{eq:phase-second-term}
\begin{aligned}
|A^{n-3}-B^{n-3}|
&\le
(n-3)|A-B|
\max\{|A|,|B|\}^{n-4}\\
&\le
(n-3)h_0t\varepsilon
(1-t+dt)^{n-4}.
\end{aligned}
\end{equation}
Substituting \eqref{eq:phase-first-term} and
\eqref{eq:phase-second-term} into \eqref{eq:phase-decomposition} yields
\[
\begin{aligned}
|e^{3i\theta}A^{n-3}-B^{n-3}|
&\le
3\varepsilon(1-t+dt)^{n-3}\\
&\quad+
(n-3)h_0t\varepsilon(1-t+dt)^{n-4}.
\end{aligned}
\]
Therefore
\[
\begin{aligned}
&|e^{2i\theta}K_2(a\mu)-K_2(h_0)|\\
&\qquad\le
nh_0\varepsilon
\left\{
3\int_0^1t^2(1-t+dt)^{n-3}\dd t
+
(n-3)h_0
\int_0^1t^3(1-t+dt)^{n-4}\dd t
\right\}.
\end{aligned}
\]
By \eqref{eq:integral-polynomial}, the right-hand side is precisely
\(E_2\), proving \eqref{eq:phase}.

We now retain the sign of the quadratic contribution.  Since
$
w_j=e^{i\theta}w'_j,
$
we have
\[
e_2(\bm w)
=
e^{2i\theta}e_2(\bm w').
\]
Consequently,
\[
\begin{aligned}
\Rea\!\left(e_2(\bm w)K_2(a\mu)\right)
&=
\Rea\!\left(
e_2(\bm w')
e^{2i\theta}K_2(a\mu)
\right)\\
&\ge
K_2(h_0)\Rea e_2(\bm w')
-
|e_2(\bm w')|E_2.
\end{aligned}
\]
Using
$
K_2(h_0)\ge K_-,
\Rea e_2(\bm w')\ge L_2,
$
and Lemma~\ref{lem:coefficients} with \(k=2\),
\[
|e_2(\bm w')|
\le
\frac{n-1}{2}v,
\]
we obtain
\begin{equation}\label{eq:quadratic-signed-lower}
\Rea\!\left(e_2(\bm w)K_2(a\mu)\right)
\ge
K_-L_2-\frac{n-1}{2}vE_2.
\end{equation}

Finally, take real parts in \eqref{eq:center-identity}.  Separating the
quadratic term in \eqref{eq:product-expansion} gives
\[
\begin{aligned}
1
\le{}&
|a\mu J|
+|1-a\mu|^n
-a^2\Rea\!\left(e_2(\bm w)K_2(a\mu)\right)\\
&\quad+
na|\mu|
\sum_{k=3}^{n-1}
a^k|e_k(\bm w)|
I_{k,n-1-k}(d).
\end{aligned}
\]
Using
$
|a\mu J|\le\mathcal C,
|1-a\mu|\le d,
$
the lower bound \eqref{eq:quadratic-signed-lower}, and
Lemma~\ref{lem:coefficients}, we obtain
\[
\begin{aligned}
1\le{}&
\mathcal C+d^n
-a^2K_-L_2
+\frac{(n-1)a^2v}{2}E_2\\
&\quad+
na|\mu|
\sum_{k=3}^{n-1}
a^k
\binom{n-1}{k}
\left(\frac{k-1}{n-2}v\right)^{k/2}
I_{k,n-1-k}(d).
\end{aligned}
\]
Finally, \eqref{eq:S-I-relation} transforms the last sum into
\[
a|\mu|
\sum_{k=3}^{n-1}
\left(
\frac{a^2v(k-1)}{n-2}
\right)^{k/2}
S_{k,n-1-k}(d),
\]
which proves \eqref{eq:signed-test}.
\end{proof}
\subsection{An analytic neighborhood of \texorpdfstring{$|\mu|=1$}{|mu|=1}}

The finite verification must not leave an unchecked neighborhood of the
limiting equality configuration \(|\mu|=1\).  The next lemma excludes
such a neighborhood uniformly for the small degrees, without any
computer-assisted verification.

\begin{lem}\label{lem:near}
There is no normalized counterexample with
$
4\le n\le13$ and $
|\mu|\ge\frac{99999}{100000}.
$
\end{lem}

\begin{proof}
Put
$
\delta=1-|\mu|.
$
Then
$
0<\delta\le10^{-5}.
$
Since \(s\ge|\mu|^2\),
\[
1-s
\le
1-|\mu|^2
=
2\delta-\delta^2
\le
2\delta.
\]

\step{Step 1. Control of \(a\) and the phase.}
Since
$
\sum_{j=1}^{n-1}(1-|q_j|^2)
=
(n-1)(1-s),
$
each summand satisfies
$
1-|q_j|^2
\le
(n-1)(1-s).
$
As \(n\le13\), we obtain
\begin{equation*}
|q_j|^2
\ge
1-(n-1)(1-s)
\ge
1-24\delta
>
\left(\frac{99}{100}\right)^2.
\end{equation*}

Moreover,
\[
(n-1)(1-s)\le24\delta<1,
\]
so the first estimate in Lemma~\ref{lem:radial-bounds} applies.  Using
\(1-s\le2\delta\), we obtain
\begin{equation}\label{eq:near-a-mu}
|a-\mu|
\le
\frac{1-s}{\sqrt{1-(n-1)(1-s)}}
\le
\frac{2\delta}{\sqrt{1-24\delta}}
<
\frac{21}{10}\delta.
\end{equation}
Since
$
\bigl|a-|\mu|\bigr|
\le
|a-\mu|,
$
and \(|\mu|=1-\delta\), \eqref{eq:near-a-mu} gives
\[
1-4\delta
\le
a
\le
1+2\delta.
\]
In particular,
$
a^2<2.
$

We next control the phase of \(\mu\).  Since
$
\bigl|\mu-|\mu|\bigr|
\le
|\mu-a|+\bigl|a-|\mu|\bigr|,
$
equation \eqref{eq:near-a-mu} yields
\[
\bigl|\mu-|\mu|\bigr|
<
\frac{21}{5}\delta.
\]
Writing \(\mu=|\mu|e^{i\theta}\), we therefore have
\[
|e^{i\theta}-1|
=
\frac{|\mu-|\mu||}{|\mu|}
<
\frac{(21/5)\delta}{1-\delta}
<
5\delta.
\]

Finally,
\[
|a|\mu|-1|
\le
|\mu|\bigl|a-|\mu|\bigr|
+
(1-|\mu|^2)
<
\frac{21}{10}\delta+2\delta,
\]
and hence
\[
|a\mu-1|
\le
a\bigl|\mu-|\mu|\bigr|
+
|a|\mu|-1|<
12\delta.
\]
We have thus proved
\begin{equation}\label{eq:near-controls}
 1-4\delta\le a\le1+2\delta,\qquad
 a^2<2,\qquad
 |e^{i\theta}-1|<5\delta,\qquad
 |a\mu-1|<12\delta.
\end{equation}

For later use, observe also that
$
1-a\mu t
=
(1-t)+t(1-a\mu).
$
Hence, by \eqref{eq:near-controls},
\begin{equation}\label{eq:near-linear-factor}
|1-a\mu t|
\le
1-t+12\delta t
\le1,
\qquad 0\le t\le1.
\end{equation}

\step{Step 2. The quadratic contribution.}
From \eqref{eq:K2},
\[
K_2(1)
=
n\int_0^1t^2(1-t)^{n-3}\dd t
=
\frac{2}{(n-1)(n-2)}.
\]
We first compare \(K_2(a\mu)\) with \(K_2(1)\).  By
\eqref{eq:near-linear-factor} and telescoping the difference of powers,
\[
\begin{aligned}
|K_2(a\mu)-K_2(1)|
&\le
n|a\mu-1|
\left(
\int_0^1t^2\dd t
+
(n-3)\int_0^1t^3\dd t
\right)\\
&=
n|a\mu-1|
\left(
\frac13+\frac{n-3}{4}
\right).
\end{aligned}
\]
Using \(n\le13\) and \eqref{eq:near-controls}, we obtain
\begin{equation}\label{eq:near-K2}
|K_2(a\mu)-K_2(1)|
\le
442\delta.
\end{equation}

Moreover,
\[
|e^{2i\theta}-1|
\le
2|e^{i\theta}-1|
<
10\delta,
\]
while
\[
K_2(1)
=
\frac{2}{(n-1)(n-2)}
\le
\frac13.
\]
Therefore \eqref{eq:near-K2} gives
\begin{equation}\label{eq:near-K2-phase}
|e^{2i\theta}K_2(a\mu)-K_2(1)|
<
450\delta.
\end{equation}

Recall that
$
e_2(\bm w)
=
e^{2i\theta}e_2(\bm w').
$
Hence
\[
\begin{aligned}
\Rea\!\left(e_2(\bm w)K_2(a\mu)\right)
&\ge
K_2(1)\Rea e_2(\bm w')
-
450\delta\,|e_2(\bm w')|.
\end{aligned}
\]
By \eqref{eq:quadratic-real},
\[
\Rea e_2(\bm w')
=
\frac{n-1}{2}v-(n-1)W,
\]
and the first bound in \eqref{eq:real-variance} gives
\[
W\le(n-2)\delta^2.
\]
Consequently,
\[
\begin{aligned}
K_2(1)\Rea e_2(\bm w')
&=
\frac{v}{n-2}
-\frac{2W}{n-2}\\
&\ge
\frac{v}{n-2}-2\delta^2.
\end{aligned}
\]
Also, Lemma~\ref{lem:coefficients} with \(k=2\) gives
\[
|e_2(\bm w')|
\le
\frac{n-1}{2}v.
\]
Combining these estimates with \eqref{eq:near-K2-phase}, we obtain
\begin{equation}\label{eq:near-quadratic}
 \Rea\!\left(e_2(\bm w)K_2(a\mu)\right)
 \ge
 \frac{v}{n-2}
 -2\delta^2
 -225(n-1)\delta v
 \ge
 -2\delta^2.
\end{equation}
For the last inequality, it is enough to note that, for
\(4\le n\le13\) and \(\delta\le10^{-5}\),
\[
\frac1{n-2}
\ge
\frac1{11}
>
225\cdot12\cdot10^{-5}
\ge
225(n-1)\delta.
\]

\step{Step 3. The higher-order coefficients.}
Set
$
\tau=a\sqrt v.
$
Since
$
v\le1-|\mu|^2\le2\delta
$
and \(a^2<2\), we have
\begin{equation}\label{eq:near-tau}
\tau
\le
2\sqrt{2\delta}
<
\frac1{100}.
\end{equation}

For the terms of order \(k\ge3\), we use the coefficient bounds from
the proof of Lemma~\ref{lem:coefficient-test}, with
\(d=12\delta\).
For \(3\le k\le n-1\), the binomial series gives
\[
\begin{aligned}
S_{k,n-1-k}(12\delta)
&\le
\sum_{j=0}^{\infty}
\binom{k+j}{k}(12\delta)^j\\
&=
(1-12\delta)^{-k-1}\\
&<2,
\end{aligned}
\]
because \(k\le n-1\le12\) and \(\delta\le10^{-5}\).

Furthermore, \eqref{eq:near-controls} gives
$
|a\mu|<2,
$
and for \(3\le k\le n-1\),
\[
\left(\frac{k-1}{n-2}\right)^{k/2}\le1.
\]
Hence the total contribution of the terms with \(k\ge3\) is bounded by
\[
\begin{aligned}
a|\mu|
\sum_{k=3}^{n-1}
\left(
\frac{a^2v(k-1)}{n-2}
\right)^{k/2}
S_{k,n-1-k}(12\delta)&\le
4\sum_{k=3}^{n-1}\tau^k<
4\sum_{k=3}^{\infty}\tau^k\\
&<
5\tau^3<
128\delta^{3/2},
\end{aligned}
\]
where the last two inequalities follow from \eqref{eq:near-tau}.  Thus
\begin{equation}\label{eq:near-higher}
a|\mu|
\sum_{k=3}^{n-1}
\left(
\frac{a^2v(k-1)}{n-2}
\right)^{k/2}
S_{k,n-1-k}(12\delta)
<
128\delta^{3/2}.
\end{equation}

We now return to the centered identity.  By \eqref{eq:core},
$
|a\mu J|
\le
|\mu|
=
1-\delta.
$
Also, by \eqref{eq:near-controls},
\[
|1-a\mu|^n
\le
(12\delta)^n
\le
(12\delta)^4.
\]
The quadratic contribution satisfies, by \eqref{eq:near-quadratic} and
\(a^2<2\),
\[
-a^2
\Rea\!\left(e_2(\bm w)K_2(a\mu)\right)
\le
4\delta^2.
\]
Combining these estimates with \eqref{eq:near-higher} in
\eqref{eq:center-identity}, we obtain
\begin{equation}\label{eq:near-final}
1
\le
1-\delta
+4\delta^2
+128\delta^{3/2}
+(12\delta)^4.
\end{equation}
But \eqref{eq:near-final} would require
\[
1
\le
4\delta
+128\sqrt\delta
+20736\delta^3.
\]
This is impossible, since
\[
4\delta
+128\sqrt\delta
+20736\delta^3
<
\frac4{10^5}
+\frac{128}{300}
+\frac{20736}{10^{15}}
<
\frac12.
\]
The contradiction proves the lemma.
\end{proof}

\section{Finite verification and completion of the proof}\label{sec:finite}

Before carrying out the finite verification, we reduce the location of the
distinguished zero to a fixed bounded interval.  This reduction is
independent of the large-degree cutoff in Proposition~\ref{prop:large}.

\begin{lem}\label{lem:compact}
Every normalized counterexample satisfies
$0<a<5.$
\end{lem}

\begin{proof}
We already know that \(a>0\).  Suppose, for contradiction, that
\(a\ge5\).  Taking \(V=1\) in the second inequality of
\eqref{eq:moment-a} gives
\[
|\mu-a|^2\le a^2(1-s).
\]
Expanding the left-hand side, we obtain
\[
a^2s
\le
2ax-|\mu|^2
\le
2a|\mu|.
\]
On the other hand, the first inequality in \eqref{eq:moment-b} gives
\[
|1-a\mu|^2\le v\le s,
\]
and hence
\[
a|\mu|
\le
1+|1-a\mu|
\le
1+\sqrt{s}.
\]
Writing \(s_0=\sqrt{s}\), we therefore have
\[
a^2s_0^2
\le
2+2s_0.
\]
Since \(a\ge5\),
\[
25s_0^2\le2+2s_0,
\]
which implies
$
s_0<1/3.
$
Consequently,
\begin{equation}\label{eq:compact-bounds}
s<\frac19,
\qquad
a^2s
\le
2a|\mu|
\le
2+2\sqrt{s}
<
\frac83.
\end{equation}

By \eqref{eq:endpoint-first-moment}, with \(V=1\),
\[
\frac1{n-1}\sum_{j=1}^{n-1}|1-aq_j|
\le
\sqrt{s}
<
\frac13.
\]
We may therefore take
\[
\kappa=\frac13,
\qquad
\ell(t)=1-\frac23t.
\]
The arithmetic--geometric mean inequality then gives
\begin{equation}\label{eq:compact-F}
|F(1)|
=
\prod_{j=1}^{n-1}|1-aq_j|
\le
3^{-(n-1)}.
\end{equation}

For the remainder term, use the linear branch of
\eqref{eq:direct-bounds}.  Since
\(\widetilde c_n<3\), \eqref{eq:compact-bounds} gives
\[
\begin{aligned}
|\mathcal R|
&\le
a^2(n-1)\widetilde c_ns
\int_0^1
t\left(1-\frac23t\right)^{n-2}\dd t\\
&<
8(n-1)
\int_0^{3/2}
t\left(1-\frac23t\right)^{n-2}\dd t\\
&=
\frac{18}{n}.
\end{aligned}
\]
Thus
\begin{equation}\label{eq:compact-R}
|\mathcal R|\le\frac{18}{n}.
\end{equation}

Finally, \eqref{eq:core}, together with
\(\Phi_{n-1}(a)\le1\), gives
$
|a\mu J|
\le
|\mu|
\prod_{j=1}^{n-1}|q_j|.
$
Since
$
|\mu|\le\sqrt{s}
$ and $
\prod_{j=1}^{n-1}|q_j|
\le
s^{(n-1)/2},
$
we obtain
\begin{equation}\label{eq:compact-J}
|a\mu J|
\le
s^{n/2}
<
3^{-n}.
\end{equation}

The normalization \eqref{eq:normalization} gives \(a^2\le n-1\).
Thus \(a\ge5\) forces \(n\ge26\).  Taking absolute values in
\eqref{eq:differential} and using \eqref{eq:origin},
\eqref{eq:compact-F}, \eqref{eq:compact-R}, and
\eqref{eq:compact-J}, we find
\[
\begin{aligned}
1
&\le
|F(1)|
+\frac{n-1}{n}|a\mu J|
+|\mathcal R|\\
&\le
3^{-(n-1)}
+3^{-n}
+\frac{18}{n}\\
&\le
3^{-25}+3^{-26}+\frac9{13}
<1,
\end{aligned}
\]
a contradiction.  Hence \(a<5\).
\end{proof}

It follows that every normalized counterexample lies in the closed
parameter box
\begin{equation}\label{eq:parameter-box}
(a,|\mu|,s)
\in
[0,5]\times[0,1]\times[0,1].
\end{equation}
The inclusion of the endpoints is deliberate.  Although admissible data
satisfy stricter inequalities, the closed box is more convenient for the
finite subdivision, and ensures that common boundaries of adjacent boxes
are covered automatically.

\Needspace{7\baselineskip}
\begin{prop}\label{prop:finite}
There is no normalized counterexample with
$
4\le n\le100000.
$
\end{prop}

\begin{proof}
The finite verification is encoded by the two certificate files
\href{https://github.com/zhangteng2000/borcea-variance/blob/83ee46a7c9711629745290ea8896805fc0239fcd/supplement/certificate_n4_to_n13.json}
{\file{certificate_n4_to_n13.json}}
and
\href{https://github.com/zhangteng2000/borcea-variance/blob/83ee46a7c9711629745290ea8896805fc0239fcd/supplement/certificate_n14_to_n100000.json}
{\file{certificate_n14_to_n100000.json}}.
The first treats the fixed degrees \(4\le n\le13\), while the second
treats \(14\le n\le100000\) in consecutive degree blocks.

Each certificate describes a finite binary partition of
\eqref{eq:parameter-box}.  Every internal node bisects one of the three
coordinate intervals at its exact midpoint, and every leaf records a
single exclusion test from Appendix~\ref{app:boxes}.  The corresponding
exact-arithmetic checkers are
\href{https://github.com/zhangteng2000/borcea-variance/blob/83ee46a7c9711629745290ea8896805fc0239fcd/supplement/verify_low.py}
{\file{verify_low.py}},
\href{https://github.com/zhangteng2000/borcea-variance/blob/83ee46a7c9711629745290ea8896805fc0239fcd/supplement/verify_mid.py}
{\file{verify_mid.py}},
and
\href{https://github.com/zhangteng2000/borcea-variance/blob/83ee46a7c9711629745290ea8896805fc0239fcd/supplement/verify_all.py}
{\file{verify_all.py}}.

The checkers reconstruct every box from its path using exact rational
arithmetic.  They reject malformed trees, verify that each internal node
has both children, and hence certify that the leaves cover the entire
parameter box.  They also verify that the degree blocks are consecutive,
with no omissions or overlaps, from degree \(4\) through degree \(100000\).

For every leaf, one of the following occurs.  Either a necessary
condition from Lemmas~\ref{lem:moments}, \ref{lem:basic-product}, and
\ref{lem:root-product}--\ref{lem:packing} fails strictly; or
Lemma~\ref{lem:near} applies; or one of the necessary estimates
\eqref{eq:direct-test}, \eqref{eq:absolute-test}, \eqref{eq:ratio-test},
\eqref{eq:coefficient-test}, and \eqref{eq:signed-test} has a right-hand
side strictly smaller than \(1\).  The interval bounds used in these
tests, including the direction of every degree-dependent estimate, are
specified in Appendix~\ref{app:boxes}.  Appendix~\ref{app:arithmetic}
shows that all numerical evaluations are performed with directed
exact-integer arithmetic, so every accepted leaf excludes the entire
box represented by that leaf.

At precision \(S=2^{100}\), the complete verification gives the
following counts.
\begin{center}
\begin{tabular}{@{}rrrr@{}}
\toprule
Degrees \(n\) & Degree blocks & Leaf boxes & Maximum depth\\
\midrule
$4$--$13$ & $10$ & $12821$ & $53$\\
$14$--$100000$ & $62$ & $20850$ & $33$\\
\midrule
Total & $72$ & $33671$ & \\
\bottomrule
\end{tabular}
\end{center}
The reported margins below concern the tests whose final rigorous upper
bound is compared with \(1\); the elementary rational contradictions
and applications of Lemma~\ref{lem:near} are verified separately.

If such an upper bound is represented as \(U/S\), the smallest positive
values of \(S-U\) in the low-degree and block verifications are,
respectively,
\begin{equation}\label{eq:certificate-margins}
 859210771742754798767913
 \quad\hbox{and}\quad
 3409997860291474113416360.
\end{equation}
Thus every accepted numerical test has strictly positive separation from
its threshold.

The same certificates also pass at precision \(2^{160}\).  The exact
commands, file hashes, and verification details are recorded in
Appendix~\ref{app:arithmetic} and in the
\href{https://github.com/zhangteng2000/borcea-variance/tree/83ee46a7c9711629745290ea8896805fc0239fcd}
{supplementary repository at commit \texttt{83ee46a7c971}}.
Since the verified partitions cover \eqref{eq:parameter-box},
Lemma~\ref{lem:compact} completes the exclusion of all normalized
counterexamples with \(4\le n\le100000\).
\end{proof}

We can now complete the proof of the main theorem.

\begin{proof}[Proof of Theorem~\ref{thm:main}]
If \(\sigma_2(P)=0\), then all zeros of \(P\) coincide, and
\eqref{eq:main} is immediate.  If the prescribed zero is multiple, then
it is itself a critical point, so there is again nothing to prove.

The cases \(n=2,3\) follow from
\cite[Theorem~3.2]{KPPSS11}.
Suppose therefore that \(n\ge4\), that \(\sigma_2(P)>0\), and that the
prescribed zero is simple.  If \eqref{eq:main} failed, the polynomial
could be normalized as in \eqref{eq:normalization}.  In the normalized
setting, \(a=0\) is impossible: otherwise
\(|\zeta_j|>1\) for every critical point, contradicting
\[
V_*
=
\frac1{n-1}\sum_{j=1}^{n-1}|\zeta_j|^2
\le
\frac{n-2}{n-1}
<1.
\]
Thus \(a>0\).

Proposition~\ref{prop:finite} excludes
$
4\le n\le100000,
$
while Proposition~\ref{prop:large} excludes
$
n\ge100001.
$
Hence no counterexample exists, and \eqref{eq:main} follows for every
degree \(n\ge2\).

The constant in \eqref{eq:main} is sharp.  Indeed, for
$
P(z)=z^n-1,
$
the centroid of the zeros is \(0\), and
$
\sigma_2(P)^2
=
\frac1n\sum_{j=1}^n|z_j|^2
=
1.
$
The only critical point is \(0\), while every zero of \(P\) has modulus
\(1\).  Thus every zero is at distance exactly
\(\sigma_2(P)=1\) from the unique critical point.
\end{proof}

\section*{Data availability}
The certificates, exact checkers, optional generators, and verification logs are available in the supplementary repository \cite{ZhaS}. This paper uses the files at
\href{https://github.com/zhangteng2000/borcea-variance/tree/83ee46a7c9711629745290ea8896805fc0239fcd}{commit \texttt{83ee46a7c971}}; the certificate digests are recorded in Appendix~\ref{app:arithmetic}. These certificates are an essential part of Proposition~\ref{prop:finite}. Reproduction requires only the Python standard library.

\appendix
\section{Bounds on a parameter box}\label{app:boxes}

This appendix gives the mathematical specification of the finite
verification used in Proposition~\ref{prop:finite}.  The labels attached
to the tests in the certificate files are program identifiers only; they
do not introduce additional mathematical assertions.  Throughout the
appendix, an \emph{admissible point} means a triple of parameters arising
from a normalized counterexample.

For convenience, write
$
r:=|\mu|.
$
A parameter box will be denoted by
\[
\mathcal B
=
[a_-,a_+]\times[r_-,r_+]\times[s_-,s_+],
\]
so that every admissible point in \(\mathcal B\) satisfies
\[
a_-\le a\le a_+,\qquad
r_-\le r=|\mu|\le r_+,\qquad
s_-\le s\le s_+.
\]
Define
\[
\begin{gathered}
v_+=s_+-r_-^2,\qquad
v_-=(s_--r_+^2)_+,\\
d_a=\max\{0,a_--r_+,r_--a_+\},\\
d_h=\max\{0,1-a_+r_+,a_-r_--1\},\qquad
\Delta_+=1-s_-.
\end{gathered}
\]
Since
$
v=s-r^2,
$
we have
\[
v_-\le v\le v_+.
\]
Moreover,
$
|a-\mu|
\ge
|a-r|
\ge
d_a
$
and
$
|1-a\mu|
\ge
|1-ar|
\ge
d_h.
$
Finally,
$
1-s\le\Delta_+.
$

All interval estimates below are interpreted under the assumption that
the box under consideration contains an admissible point.

\subsection{Refinements for a fixed small degree}\label{app:low}

In this subsection \(4\le n\le13\) is fixed.  We begin with the
Schoenberg bound in \eqref{eq:critical-moments}
$
V_0:=\frac{n-2}{n-1}.
$
If
$
t_*:=1-(n-1)\Delta_+>0,
$
define
\[
D_+^2
=
\frac{\Delta_+^2}{t_*},
\qquad
V_0
\leftarrow
\min\left\{
\frac{n-2}{n-1},
1+\frac{\Delta_+}{t_*}-a_-^2
\right\}.
\]
Indeed, Lemma~\ref{lem:radial-bounds} gives
\[
|a-\mu|^2
\le
\frac{(1-s)^2}{1-(n-1)(1-s)}
\le
D_+^2,
\]
and
\[
V_*
\le
1+
\frac{1-s}{1-(n-1)(1-s)}
-a^2
\le
1+\frac{\Delta_+}{t_*}-a_-^2.
\]
If \(t_*\le0\), the quantity \(D_+^2\) is left unspecified and we retain
only the original bound \(V_0=(n-2)/(n-1)\).  A negative value of
\(V_0\) excludes the box immediately.

For the root-product estimates, put
\[
\begin{gathered}
B_+
=
\frac{n}{n-1}(1+a_+^2),\qquad
C_+
=
\frac{n}{(n-1)^2}(n-1-a_-^2)_+,\\
g_-
=
\max\left\{
0,\,
a_--\frac{1+a_+^2}{2}r_+,\,
\frac{1+a_-^2}{2}r_--a_+
\right\},
\end{gathered}
\]
and define
\[
\gamma_-=
\begin{cases}
\displaystyle
\frac{n-1}{3}
\left(\frac{n}{n-1}\right)^2
\frac{g_-^2}{C_+},
& C_+>0,\\[8pt]
0,
& C_+=0.
\end{cases}
\]
Finally, set
\[
L
=
\frac{n^2(1+\gamma_-)}{B_+^{n-1}}.
\]
By \eqref{eq:ABC},
\[
B_0\le B_+,\qquad C_0\le C_+,
\]
and
\[
\left|A_0-\frac{B_0\mu}{2}\right|
=
\frac{n}{n-1}
\left|
a-\frac{1+a^2}{2}\mu
\right|
\ge
\frac{n}{n-1}g_-.
\]
Lemma~\ref{lem:root-product} therefore gives
\[
L
\le
\prod_{j=1}^{n-1}|q_j|^2.
\]
If \(C_0=0\), the covariance gap in
Lemma~\ref{lem:root-product} vanishes.  Thus no division by zero occurs:
a box containing admissible data with \(C_0=0\) cannot simultaneously
produce a strictly positive lower bound \(g_-\).

Next, define
\[
d_s
=
\max\{0,a_-s_--r_+,r_--a_+s_+\}.
\]
Since
$
|\mu-as|
\ge
|r-as|
\ge
d_s,
$
Lemma~\ref{lem:radial-loss} gives the following refinement.  If
\(s_+V_0>0\), set
\[
\lambda_-
=
\frac{(n-1)d_s^2}{2s_+V_0};
\]
otherwise set \(\lambda_-=0\).  Then
\begin{equation}\label{eq:box-loss}
\prod_{j=1}^{n-1}|q_j|^2
\le
s_+^{n-1}e^{-\lambda_-}.
\end{equation}
When \(V_0=0\), the choice \(\lambda_-=0\) uses only the
arithmetic--geometric mean bound and therefore involves no division by
\(V_0\).

Suppose now that
$
0<L\le s_+^{n-1}.
$
The checker chooses \(\eta\) by \(30\) exact bisections on
\([0,s_+]\), retaining a lower endpoint for which \eqref{eq:eta} holds
with \(s_+\) in place of \(s\).  Whenever this lower endpoint is
positive, Lemma~\ref{lem:packing} supplies \(k,z,M,T\).  Let
\(M_+\ge M\) be an upward-rounded upper bound.  We then refine
\begin{equation}\label{eq:box-packed}
D_+^2
\leftarrow
\min\left\{
D_+^2,
\left(\frac{M_+}{n-1}\right)^2
\right\},
\qquad
V
=
\min\{V_0,T-a_-^2\}.
\end{equation}
If \(D_+^2\) was previously unspecified, the first minimum in
\eqref{eq:box-packed} is interpreted as the newly available bound
\((M_+/(n-1))^2\).  If the packing refinement is unavailable, we simply
set \(V=V_0\).

The product loss \eqref{eq:box-loss} continues to use the original
quantity \(V_0\), rather than the subsequently refined \(V\).  Thus the
order of the preceding refinements introduces no circularity.

For later use, define the positive Taylor polynomial
\begin{equation}\label{eq:T32}
T_{32}(x)
:=
\sum_{j=0}^{32}\frac{x^j}{j!},
\qquad x\ge0.
\end{equation}
Then
\[
T_{32}(x)\le e^x.
\]

For clarity, Table~\ref{tab:primitive} lists the elementary exclusion
tests used for a fixed degree.  The tests involving only the original
Schoenberg bound may use \(V_0\), whereas
\file{moment_A_packed} and \file{moment_B_packed} use the refined value
\(V\).  Every strict inequality displayed in the table excludes the
entire box.

\begin{table}[htbp]
\centering\small
\caption{Elementary exclusion tests for a fixed degree.}
\label{tab:primitive}
\begin{tabular}{@{}p{.34\textwidth}p{.62\textwidth}@{}}
\toprule
Identifier & Sufficient condition for exclusion\\
\midrule
\file{degree_bound}
& $a_-^2>n-1$\\
\file{variance_nonnegative}
& $v_+<0$\\
\file{inverse_moment}
& $(a_+^2+V_0)s_+<1$\\
\file{moment_A}
& $d_a^2>(a_+^2+V_0-1)_+(1-s_-)$\\
\file{moment_B}
& $d_h^2>V_0v_+$\\
\file{product}
& $(B_+s_+)^{n-1}<n^2$\\
\file{finite_radial}
& $d_a^2(1-(n-1)\Delta_+)>\Delta_+^2$\\
\file{product_stable}
& $(B_+s_+)^{n-1}<n^2T_{32}(\lambda_-)$\\
\file{product_root_stable}
& $(B_+s_+)^{n-1}<n^2(1+\gamma_-)T_{32}(\lambda_-)$\\
\file{product_radial}
& $L>1$, or $0<L\le1$ and
  $(n-1)^2d_a^2L>(1-L)^2$\\
\file{product_logcap}
& $(n-1)d_a>M_+$ when \(M_+\) is available\\
\file{critical_variance_negative}
& $V_0<0$ or $V<0$\\
\file{moment_A_packed}
& $d_a^2>(a_+^2+V-1)_+(1-s_-)$\\
\file{moment_B_packed}
& $d_h^2>Vv_+$\\
\file{near_uniform}
& $r_-\ge99999/100000$\\
\bottomrule
\end{tabular}
\end{table}

The test \file{product_radial} follows from
\eqref{eq:product-radial}.  Indeed, the function
$
x\longmapsto\frac{1-x}{\sqrt{x}}
$
is decreasing on \((0,1]\), and
$
L\le\prod_{j=1}^{n-1}|q_j|^2.
$
The test \file{finite_radial} is relevant only when
\[
1-(n-1)\Delta_+>0.
\]
Finally, \file{near_uniform} is exactly Lemma~\ref{lem:near}; it is not
a numerical approximation to that lemma.

For the remaining tests assume
$
V\ge0,  v_+\ge0,
$
and define
\[
H
=
\max\left\{
r_-^2+(1-V)(1-s_+),\,
1-(a_+^2+V)v_+
\right\},
\qquad
A=a_-^2s_-.
\]
If \(D_+^2\) is available, include
\[
a_-^2(1-s_+)+r_-^2-D_+^2
\]
as a third entry in the maximum defining \(H\).  Lemma~\ref{lem:moments}
then gives
\[
\rho\ge H.
\]
The third entry follows from the identity
\[
\rho
=
a^2(1-s)+r^2-|a-\mu|^2.
\]
Since
$
\beta_s(t)
=
1-\rho t-a^2s\,t(1-t),
$
we obtain the box-wise majorant
\[
\beta_s(t)
\le
b(t)
:=
1-Ht-At(1-t).
\]
If \(H>1\), or if \(b\) has a negative minimum on \([0,1]\), the box is
excluded; this is the test \file{beta_negative}.  Otherwise the checker
requires
\[
0\le H\le1,
\qquad
b(t)\ge0
\quad(0\le t\le1).
\]
When \(A>0\), the only possible interior minimum occurs at
$
t=\frac{H+A}{2A},
$
and its value is
$
1-\frac{(H+A)^2}{4A}.
$
Thus this check is entirely rational.

Define
\begin{equation}\label{eq:box-endpoint}
\kappa^2
=
\min\{1-H,Vs_+\},
\qquad
d^2
=
\max\left\{
0,\,
\min\{Vv_+,1-H-a_-^2v_-\}
\right\},
\qquad
\ell(t)=1-t+\kappa t.
\end{equation}
The first definition ensures \eqref{eq:kappa}.  For the second, recall
that
$
|1-a\mu|^2\le Vv
$
and
$
|1-a\mu|^2
=
\beta_s(1)-a^2v.
$
Hence, on an admissible box,
\[
|1-a\mu|^2
\le
\min\{Vv_+,1-H-a_-^2v_-\}.
\]
Both quantities inside this minimum are then nonnegative, so the outer
maximum with \(0\) does not weaken the asserted upper bound.

Let \(\Phi_+\) denote the maximum of \(\Phi_{n-1}\) over
$
[a_-,a_+]\cap[0,\sqrt{n-1}].
$
A box for which this intersection is empty has already been excluded by
the degree bound.  By Lemma~\ref{lem:basic-product}, the maximum is
attained at \(1\) if \(1\in[a_-,a_+]\), at \(a_+\) if \(a_+<1\), and
at \(a_-\) if \(a_->1\).  Thus no irrational interval endpoint is
needed.

Combining \eqref{eq:core} and \eqref{eq:box-loss}, and using
$
e^{\lambda_-/2}\ge T_{32}(\lambda_-/2),
$
gives the interval upper bound
\[
\mathcal C_+
=
\frac{
\Phi_+r_+s_+^{(n-1)/2}
}{
T_{32}(\lambda_-/2)
}
\ge
|a\mu J|.
\]
A box entirely outside \(a^2\le n-1\) has already been excluded, so
every evaluation of \(\Phi_+\) has a nonnegative base.

The tests \file{direct} and \file{direct_retained} evaluate
\eqref{eq:direct-test}, using \eqref{eq:direct-bounds} alone or together
with \eqref{eq:retained}.  The tests \file{center_absolute} and
\file{center_ratio} evaluate \eqref{eq:absolute-test} and
\eqref{eq:ratio-test}; these tests are used only for \(n\ge5\).
All integrals are bounded by \eqref{eq:chord-one} and
\eqref{eq:chord-two}, with the constants in \eqref{eq:constants}.

For a mesh interval \(I=[l,u]\), define
\begin{equation*}
d_I
:=
\max\{0,1-a_+r_+u,a_-r_-l-1\}.
\end{equation*}
Then
\[
|1-a\mu t|
\ge d_I
\qquad(t\in I)
\]
for every admissible point in the box.  The nonsingular estimate
\eqref{eq:interpolation-second} is always retained; the estimate
\eqref{eq:interpolation-first} is used only when \(d_I>0\).

Whenever a centered integral test is evaluated, let
$
\mathcal I_{2,+}
$
denote the rigorous upper bound, obtained from the chord estimates, for
the corresponding integral
$
\int_0^1t^2\mathcal U(t)\dd t.
$
We then write
\begin{equation*}
K_+
:=
na_+^3r_+\mathcal I_{2,+}.
\end{equation*}
By \eqref{eq:K-definition},
\[
K\le K_+.
\]
Thus \(K_+\) is simply the box-wise upper bound for the quantity \(K\)
appearing in Lemma~\ref{lem:integral-tests}.

The test \file{center_coeff} uses \eqref{eq:coefficient-test} with
\(\mathcal C_+\), the value \(d\) from \eqref{eq:box-endpoint}, and
upper endpoints in every nonnegative coefficient.

The final test \file{center_positive} uses Lemma~\ref{lem:phase} in the
form \eqref{eq:signed-test}.  It is used only when
\[
r_->0,
\qquad
h_+:=a_+r_+\le1.
\]
Set
$
h_-:=a_-r_-.
$
If the packing refinement above is unavailable, put \(\eta=0\).
Otherwise \(\eta\) denotes the positive lower endpoint produced by the
packing bisection.  Define
\[
\begin{gathered}
t_0
=
\max\{0,\eta,1-(n-1)(1-s_-)\},\\
R_-
=
\max\left\{
-1,\,
(n-1)r_--(n-2),\,
\frac{(n-1)^2r_-^2+t_0-(n-2)^2}
{2(n-1)r_-}
\right\},\\
W_+
=
\min\left\{
(n-2)(1-r_-)^2,\,
(1-r_-)(r_+-R_-)
\right\},\\
L_2
=
\frac{n-1}{2}(v_--2W_+).
\end{gathered}
\]
Each of the three lower bounds in \eqref{eq:real-lower} is nondecreasing
in \(r=|\mu|>0\) when \(0\le t_0\le1\) and \(n\ge4\).
Lemma~\ref{lem:real-variance} therefore gives
\[
W\le W_+.
\]
The checker invokes this test only when
$
L_2>0.
$

Finally set
\[
x_-
=
\frac{H+a_-^2s_-}{2a_+},
\qquad
\varepsilon
=
\sqrt{\max\{0,2(1-x_-/r_+)\}},
\qquad
K_-
=
nh_-I_{2,n-3}(1-h_+).
\]
The identity
$
2ax=\rho+a^2s
$
gives
$
x\ge x_-.
$
Since
$
|e^{i\theta}-1|^2
=
2\left(1-\frac{x}{r}\right),
$
we obtain
$
|e^{i\theta}-1|
\le\varepsilon.
$
Moreover,
\[
K_2(a|\mu|)
=
na|\mu|
I_{2,n-3}(1-a|\mu|)
\ge
nh_-I_{2,n-3}(1-h_+)
=
K_-,
\]
so that
\[
0\le K_-\le K_2(a|\mu|).
\]
In \eqref{eq:phase}, \(h_0\) is replaced by the box upper bound \(h_+\)
when estimating \(E_2\) from above.  In \eqref{eq:signed-test}, the
favorable quadratic term is bounded from above by
$
-a_-^2K_-L_2.
$
Because this quantity is subtracted, it is rounded in the direction
appropriate for a lower bound on the positive factor
\(a^2K_-L_2\), as explained in Appendix~\ref{app:arithmetic}.

\subsection{Several consecutive degrees in one block}\label{app:mid}

We now consider a degree block
\[
n-1\in[M,N],
\qquad
13\le M\le N.
\]
The block checker uses only Lemma~\ref{lem:moments},
Lemma~\ref{lem:basic-product}, and the integral estimates of
Section~\ref{sec:integral}.  Set
$
V
=
1-\frac1N,
v_+=s_+-r_-^2,
$
and
\[
H
=
\max\left\{
r_-^2+\frac{1-s_+}{N},\,
1-(a_+^2+V)v_+
\right\},
\qquad
A=a_-^2s_-.
\]
Since
\[
V_*
\le
1-\frac1{n-1}
\le
1-\frac1N
=
V,
\]
these are valid uniform bounds throughout the degree block.

Use
$
b(t)=1-Ht-At(1-t),
\kappa^2=\min\{1-H,Vs_+\},
$
together with the same nonnegativity checks as in
Appendix~\ref{app:low}.  The elementary tests are the degree bound, the
nonnegative-variance test, the two covariance tests, the inverse-moment
test, the negative-\(b\) test, and the product test described below.

Lemma~\ref{lem:basic-product} gives the uniform bound
\[
|a\mu J|
\le
\mathcal C_+
:=
\max_{\substack{a\in[a_-,a_+]\\a^2\le N}}
\Phi_N(a)\,
r_+s_+^{M/2}.
\]
The maximum of \(\Phi_N\) is evaluated by the same endpoint rule used in
Appendix~\ref{app:low}.

For the elementary product test, put
\[
B
=
\frac{M+1}{M}(1+a_+^2)s_+.
\]
If \(B\le1\), then the actual base in \eqref{eq:product-basic} is at
most \(1\), whereas \eqref{eq:product-basic} requires its positive power
to be at least \(n^2>1\).  Thus \(B\le1\) already excludes the box.

If \(B>1\), then increasing the exponent gives an upper bound, and the
condition
\[
B^N<(M+1)^2
\]
again contradicts \eqref{eq:product-basic}.  This is the block product
test.

For the direct integral estimate, we use the coefficients
\[
\min\{1,(5/3)s_+\},
\qquad
\min\{1,3s_+\},
\]
the powers
\[
b^{(M-1)/2},
\qquad
\ell^{M-1},
\]
and the prefactor \(Na_+^2\).  The two endpoint terms in
\eqref{eq:direct-test} are bounded by
\[
\kappa^M,
\qquad
\frac{N}{N+1}\mathcal C_+.
\]
These are valid upper bounds because
\[
0\le b,\ell,\kappa\le1.
\]

For \file{direct_retained}, let \(I=[l,u]\) be a mesh interval and set
$
 \delta_I:=(1-a_+u)_+.
$
For every admissible point and every \(t\in I\),
\[
 |1-atq_j|
 \ge (1-at|q_j|)_+
 \ge \delta_I,
 \qquad 1\le j\le n-1.
\]
Consequently,
\[
 b(t)\ge\delta_I^2,
 \qquad
 \ell(t)\ge\delta_I.
\]
For fixed \(t\), the functions
\[
 \frac{(n-1)b(t)-\delta_I^2}{n-2},
 \qquad
 \frac{(n-1)\ell(t)-\delta_I}{n-2}
\]
are nonincreasing in \(n-1\). We therefore enlarge their bases by
setting \(n-1=M\), and write
\[
 B_I(t):=
 \left(\frac{Mb(t)-\delta_I^2}{M-1}\right)_+,
 \qquad
 L_I(t):=
 \left(\frac{M\ell(t)-\delta_I}{M-1}\right)_+.
\]

These enlarged bases need not be at most \(1\). The corresponding
powers are therefore bounded by
\begin{equation}\label{eq:block-retained-powers}
 \max\{B_I(t)^{(M-1)/2},B_I(t)^{(N-1)/2}\},
 \qquad
 \max\{L_I(t)^{M-1},L_I(t)^{N-1}\}.
\end{equation}
Both expressions are convex functions of \(t\): the bases are
nonnegative and convex, all exponents are at least one, and taking
the maximum preserves convexity. We apply \eqref{eq:chord-one}
to these two majorants separately.

For the centered integral estimates, the checker uses the four uniform
coefficients
\[
\frac{5N}{6d_I},
\qquad
\frac{3N}{2d_I},
\qquad
\frac{3N(N-1)}2,
\qquad
\frac{9N(N-1)}2.
\]
The first two are omitted when \(d_I=0\).  The corresponding powers are
\[
b^{(M-1)/2},
\qquad
\ell^{M-1},
\qquad
b^{(M-2)/2},
\qquad
\ell^{M-2}.
\]

In the definition of \(K\), the prefactor satisfies
\[
na^3r
\le
(N+1)a_+^3r_+.
\]
Thus \(K_+\) denotes the upper bound obtained by multiplying
\((N+1)a_+^3r_+\) by the rigorous chord bound for the relevant integral
\(\int t^2\mathcal U(t)\dd t\).

For \eqref{eq:absolute-test}, we use
\[
(Vv)^{n/2}
\le
(Vv_+)^{(M+1)/2}.
\]
For \eqref{eq:ratio-test}, we use
\[
(1+r_+)
\left\{
V^{(M+1)/2}
(1-r_-^2)^{(M-1)/2}
+
K_+
\right\}.
\]
All bases in these centered powers lie in \([0,1]\), so decreasing the
exponent enlarges the expression.  The powers in
\eqref{eq:block-retained-powers} are the deliberate exception, because
their bases are not known a priori to be at most \(1\).

\section{Exact arithmetic and reproducibility}\label{app:arithmetic}

\subsection{Certificates and coverage}

The two certificate files
\href{https://github.com/zhangteng2000/borcea-variance/blob/83ee46a7c9711629745290ea8896805fc0239fcd/supplement/certificate_n4_to_n13.json}
{\file{certificate_n4_to_n13.json}}
and
\href{https://github.com/zhangteng2000/borcea-variance/blob/83ee46a7c9711629745290ea8896805fc0239fcd/supplement/certificate_n14_to_n100000.json}
{\file{certificate_n14_to_n100000.json}}
encode the partitions used in Proposition~\ref{prop:finite}.

A degree block records an inclusive range \([M,N]\) for \(n-1\), a
positive mesh-refinement parameter, and a finite list of leaf paths
together with their test identifiers.  The initial parameter box is
\eqref{eq:parameter-box}.

Each pair of characters in a path records a coordinate index in
\(\{0,1,2\}\) and a child index in \(\{0,1\}\).  The corresponding
coordinate interval is bisected at its exact midpoint.  Both children
are represented by closed intervals, so adjacent leaves may share a
boundary but have disjoint interiors.

The checker inserts all paths into a tree.  It rejects repeated leaves,
a leaf with descendants, conflicting split coordinates, or an internal
node missing either child.  All endpoints are reconstructed as exact
rational numbers.  Induction on this finite tree proves that the leaves
cover the initial box.  The checker also verifies that consecutive
degree blocks have neither gaps nor overlaps.  Thus no information from
the floating-point certificate generator is accepted merely because it
appears in a certificate.

For a block beginning at \(n-1=M\), set
\[
L_{\mathrm{mesh}}
=
\max\{4,\lceil\log_2M\rceil+2\}.
\]
The mesh first divides
$
[0,2^{-L_{\mathrm{mesh}}}]
$
into \(q\) equal subintervals, and then divides each interval
\[
[2^{-j},2^{-j+1}],
\qquad
j=L_{\mathrm{mesh}},\ldots,1,
\]
into \(q\) equal subintervals.  The small-degree certificates use
\(q=32\), the degree-\(14\) block uses \(q=16\), and all remaining
blocks use \(q=4\).  The checker verifies directly that the resulting
rational mesh points increase from \(0\) to \(1\).

The integral bounds are exactly \eqref{eq:chord-one} and
\eqref{eq:chord-two}; no floating-point quadrature is used in the
verification.

\subsection{Arithmetic operations}

Fix
$
S=2^\nu,
$
where the primary verification uses \(\nu=100\).  Here \(\nu\) denotes
the bit precision and is unrelated to the majorant \(b(t)\) used
earlier.

A nonnegative integer \(U\) represents the rigorous upper bound \(U/S\).
For a nonnegative rational number \(p/q\), its upward enclosure is
$
\left\lceil\frac{Sp}{q}\right\rceil.
$
Multiplication of upper bounds is implemented by
\begin{equation}\label{eq:rounding}
(U,V)
\longmapsto
\left\lceil\frac{UV}{S}\right\rceil,
\qquad
\left(\frac pq,U\right)
\longmapsto
\left\lceil\frac{pU}{q}\right\rceil.
\end{equation}

For a square root, if \(U/S\) is the current upper bound, the checker
chooses the smallest nonnegative integer \(W\) satisfying
$
W^2\ge US.
$
Then
\[
\frac WS\ge\sqrt{\frac US}.
\]
The integer \(W\) is obtained from the integer square root of \(US\),
with one added unless \(US\) is a perfect square.

Integer powers are evaluated by repeated squaring using
\eqref{eq:rounding}.
For a half-integer power \(x^{k+1/2}\), the checker multiplies
an upward enclosure of \(x^k\) by an upward enclosure of
\(\sqrt{x}\), again rounding upward.

The favorable negative term in \eqref{eq:signed-test} requires the
opposite rounding direction.  If \(c\ge0\) is a rational lower bound
for the positive quantity \(a^2K_-L_2\), the checker subtracts
$
\lfloor Sc\rfloor,
$
rather than \(\lceil Sc\rceil\), from the upper-bound numerator.  Thus
directed rounding can never make the final upper bound artificially
small.

Exponential factors are handled using the positive Taylor polynomial
\(T_{32}\) from \eqref{eq:T32}; no floating-point exponential or
logarithm is used in an acceptance decision.

The interpretation of every stored integer as a rigorous upper or lower
bound is preserved inductively by these operations.  A main numerical
test is accepted only when its final upper-bound numerator \(U\)
satisfies
$
U<S.
$
Thus the positive margins in \eqref{eq:certificate-margins} represent
strict mathematical inequalities, not numerical tolerances.  The
elementary exclusions are comparisons of exact rational numbers.
Timing information in the verification logs plays no role in
acceptance.

\subsection{Reproduction}

The essential supplementary files are
\href{https://github.com/zhangteng2000/borcea-variance/blob/83ee46a7c9711629745290ea8896805fc0239fcd/supplement/certificate_n4_to_n13.json}
{\file{certificate_n4_to_n13.json}},
\href{https://github.com/zhangteng2000/borcea-variance/blob/83ee46a7c9711629745290ea8896805fc0239fcd/supplement/certificate_n14_to_n100000.json}
{\file{certificate_n14_to_n100000.json}},
\href{https://github.com/zhangteng2000/borcea-variance/blob/83ee46a7c9711629745290ea8896805fc0239fcd/supplement/verify_low.py}
{\file{verify_low.py}},
\href{https://github.com/zhangteng2000/borcea-variance/blob/83ee46a7c9711629745290ea8896805fc0239fcd/supplement/verify_mid.py}
{\file{verify_mid.py}},
and
\href{https://github.com/zhangteng2000/borcea-variance/blob/83ee46a7c9711629745290ea8896805fc0239fcd/supplement/verify_all.py}
{\file{verify_all.py}}.

The checker
\href{https://github.com/zhangteng2000/borcea-variance/blob/83ee46a7c9711629745290ea8896805fc0239fcd/supplement/verify_low.py}
{\file{verify_low.py}}
treats \(4\le n\le13\) using Appendix~\ref{app:low}, while
\href{https://github.com/zhangteng2000/borcea-variance/blob/83ee46a7c9711629745290ea8896805fc0239fcd/supplement/verify_mid.py}
{\file{verify_mid.py}}
treats \(14\le n\le100000\) using Appendix~\ref{app:mid}.
The former imports the integer-arithmetic routines and mesh construction
from the latter, but neither checker imports a certificate generator.

The wrapper
\href{https://github.com/zhangteng2000/borcea-variance/blob/83ee46a7c9711629745290ea8896805fc0239fcd/supplement/verify_all.py}
{\file{verify_all.py}}
invokes both checkers, verifies their degree endpoints, and records the
combined result.  The arithmetic self-tests are contained in
\href{https://github.com/zhangteng2000/borcea-variance/blob/83ee46a7c9711629745290ea8896805fc0239fcd/supplement/self_test.py}
{\file{self_test.py}}.

Use Python~3.10 or later, without optimization flags.
From the root of the repository at the pinned commit, run
\begin{verbatim}
cd supplement
python3 self_test.py
python3 verify_all.py --bits 100
python3 verify_all.py --bits 160
\end{verbatim}
Only the Python standard library is required.
The optional certificate generators are not used by these commands.  The
optional certificate generators require NumPy, but they are not part of
the acceptance procedure.

The two certificate files are identified by the following SHA--256
digests; line breaks below are not part of the digests.
\begin{center}\small
\begin{tabular}{@{}l@{}}
\href{https://github.com/zhangteng2000/borcea-variance/blob/83ee46a7c9711629745290ea8896805fc0239fcd/supplement/certificate_n4_to_n13.json}
{\file{certificate_n4_to_n13.json}}\\
\texttt{4334dd70fe5317040124115239f92ecc}\\
\texttt{4e220cbf6c3576407396c3ea9df19f8b}\\[4pt]
\href{https://github.com/zhangteng2000/borcea-variance/blob/83ee46a7c9711629745290ea8896805fc0239fcd/supplement/certificate_n14_to_n100000.json}
{\file{certificate_n14_to_n100000.json}}\\
\texttt{f9131604a764777836afed040bd73b1f}\\
\texttt{8ca24c5696022460b5547cb509af766d}\\
\end{tabular}
\end{center}

For the first certificate, the numbers of leaves in degrees
\(4,\ldots,13\) are
\[
532,\ 1103,\ 1594,\ 2540,\ 1753,\ 1420,\ 1164,\ 1002,\ 904,\ 809,
\]
respectively.  Ten leaves, five in degree \(4\) and five in degree \(5\),
invoke Lemma~\ref{lem:near}.  Every other leaf is checked by one of the
exact rational or integer tests specified in Appendix~\ref{app:boxes}.

The complete runs at \(100\) and \(160\) bits verify the same
partitions.  They are checks at two different precisions, rather than
two independent implementations of the full argument.  Neither run is
a proof-assistant formalization of the analytic lemmas.  The arithmetic
self-tests verify elementary identities, selected analytic constants,
and rejection of malformed or deliberately false certificates; these
tests supplement, but do not replace, any mathematical argument.

\subsection{Elementary analytic constants}

For completeness, the exponential comparisons used in
Proposition~\ref{prop:large} can also be verified entirely by rational
arithmetic.  If \(x\ge0\) and \(x<k+2\), then
\begin{equation}\label{eq:exp-series-bounds}
 \sum_{j=0}^k\frac{x^j}{j!}
 \le
 e^x
 \le
 \sum_{j=0}^k\frac{x^j}{j!}
 +
 \frac{x^{k+1}}{(k+1)!}
 \frac1{1-x/(k+2)}.
\end{equation}
The upper bound follows because, beginning with the
\((k+1)\)-st term, the ratio of successive terms in the exponential
series is at most \(x/(k+2)<1\).  Hence the remaining tail is bounded
by a geometric series.

Formula~\eqref{eq:exp-series-bounds} with \(k=20\) verifies the four
exponential comparisons used in Step~2 of
Proposition~\ref{prop:large}.  It also yields
\[
e<\frac{11}{4}<\frac{25}{9},
\]
which implies the uniform bounds following \eqref{eq:constants}.
The remaining estimates in Proposition~\ref{prop:large} use only the
displayed rational inequalities,
\[
1-u\le e^{-u}\qquad(u\ge0),
\]
and
\[
e^x\ge\frac{x^6}{720}\qquad(x\ge0).
\]
Thus the large-degree argument is analytic throughout its stated range;
none of its estimates is obtained by extrapolating the finite
verification.

\begin{thebibliography}{KPPSS11}

\bibitem[Ada26]{Ada26}
T.~Adamczewski,
\emph{Smale's mean value conjecture ($K=1$)},
aiXiv:2609.01654 [math.AG], 2026.
\url{https://aixiv.online/abs/2609.01654}.
Formal Lean proof available at
\url{https://github.com/tadamcz/mean-value-problem}.

\bibitem[BD00]{BD00}
R.~Bhatia and C.~Davis,
\emph{A better bound on the variance},
Amer. Math. Monthly \textbf{107} (2000), no.~4, 353--357.
\doi{10.1080/00029890.2000.12005203}.

\bibitem[Bor96a]{Bor96a}
I.~Borcea,
\emph{On the Sendov conjecture for polynomials with at most six distinct roots},
J. Math. Anal. Appl. \textbf{200} (1996), no.~1, 182--206.
\doi{10.1006/jmaa.1996.0198}.

\bibitem[Bor96b]{Bor96b}
I.~Borcea,
\emph{The Sendov conjecture for polynomials with at most seven distinct zeros},
Analysis \textbf{16} (1996), no.~2, 137--160.
\doi{10.1524/anly.1996.16.2.137}.

\bibitem[Bor98]{Bor98}
J.~Borcea,
\emph{Two approaches to Sendov's conjecture},
Arch. Math. (Basel) \textbf{71} (1998), no.~1, 46--54.
\doi{10.1007/s000130050232}.

\bibitem[Bor06]{Bor06}
J.~Borcea,
\emph{Maximal and linearly inextensible polynomials},
Math. Scand. \textbf{99} (2006), no.~1, 53--75.
\doi{10.7146/math.scand.a-14999}.

\bibitem[Bor07]{Bor07}
J.~Borcea,
\emph{Equilibrium points of logarithmic potentials induced by positive charge distributions.~I. Generalized de Bruijn--Springer relations},
Trans. Amer. Math. Soc. \textbf{359} (2007), no.~7, 3209--3237.
\doi{10.1090/S0002-9947-07-04251-1}.

\bibitem[CN06]{CN06}
W.~S.~Cheung and T.~W.~Ng,
\emph{A companion matrix approach to the study of zeros and critical points of a polynomial},
J. Math. Anal. Appl. \textbf{319} (2006), no.~2, 690--707.
\doi{10.1016/j.jmaa.2005.06.071}.

\bibitem[BS99]{BS99}
M.~G.~de Bruin and A.~Sharma,
\emph{On a Schoenberg-type conjecture},
J. Comput. Appl. Math. \textbf{105} (1999), no.~1--2, 221--228.
\doi{10.1016/S0377-0427(99)00013-8}.

\bibitem[GGTW25]{GGTW25}
B.~Georgiev, J.~G\'omez-Serrano, T.~Tao, and A.~Z.~Wagner,
\emph{Mathematical exploration and discovery at scale},
preprint, 2025.
\arxiv{2511.02864}.

\bibitem[KPPSS11]{KPPSS11}
D.~Khavinson, R.~Pereira, M.~Putinar, E.~B.~Saff, and S.~Shimorin,
\emph{Borcea's variance conjectures on the critical points of polynomials},
in \emph{Notions of positivity and the geometry of polynomials}
(P.~Br\"and\'en, M.~Passare, and M.~Putinar, eds.),
Trends Math., Birkh\"auser, Basel, 2011, pp.~283--309.
\doi{10.1007/978-3-0348-0142-3_16}.

\bibitem[KT16]{KT16}
O.~Kushel and M.~Tyaglov,
\emph{Circulants and critical points of polynomials},
J. Math. Anal. Appl. \textbf{439} (2016), no.~2, 634--650.
\doi{10.1016/j.jmaa.2016.03.005}.

\bibitem[LXZ21]{LXZ21}
M.~Lin, M.~Xie, and J.~Zhang,
\emph{Remarks on circulant matrices and critical points of polynomials},
J. Math. Anal. Appl. \textbf{502} (2021), no.~1, Paper No.~125233.
\doi{10.1016/j.jmaa.2021.125233}.

\bibitem[Mal05]{Mal05}
S.~M.~Malamud,
\emph{Inverse spectral problem for normal matrices and the Gauss--Lucas theorem},
Trans. Amer. Math. Soc. \textbf{357} (2005), no.~10, 4043--4064.
\doi{10.1090/S0002-9947-04-03649-9}.

\bibitem[Maz26]{Maz26}
L.~Mazur,
\emph{A computer-assisted proof of Sendov's conjecture},
public version, August~5, 2026, 19~pp., ProofAtlas,
\url{https://www.proofatlas.ai/papers/sendov-conjecture/SENDOV_CONJECTURE_PROOF_AUGUST_5_2026.pdf}.

\bibitem[MS69]{MS69}
A.~Meir and A.~Sharma,
\emph{On Ilyeff's conjecture},
Pacific J. Math. \textbf{31} (1969), no.~2, 459--467.
\doi{10.2140/pjm.1969.31.459}.

\bibitem[Per03]{Per03}
R.~Pereira,
\emph{Differentiators and the geometry of polynomials},
J. Math. Anal. Appl. \textbf{285} (2003), no.~1, 336--348.
\doi{10.1016/S0022-247X(03)00465-7}.

\bibitem[PR72]{PR72}
D.~Phelps and R.~S.~Rodriguez,
\emph{Some properties of extremal polynomials for the Ilieff conjecture},
Kodai Math. Sem. Rep. \textbf{24} (1972), no.~2, 172--175.
\doi{10.2996/kmj/1138846519}.

\bibitem[Sch77]{Sch77}
G.~Schmeisser,
\emph{On Ilieff's conjecture},
Math. Z. \textbf{156} (1977), no.~2, 165--173.
\doi{10.1007/BF01178761}.

\bibitem[Sch86]{Sch86}
I.~J.~Schoenberg,
\emph{A conjectured analogue of Rolle's theorem for polynomials with real or complex coefficients},
Amer. Math. Monthly \textbf{93} (1986), no.~1, 8--13.
\doi{10.1080/00029890.1986.11971734}.

\bibitem[Tan25]{Tan25}
Q.~Tang,
\emph{Schoenberg type inequalities},
preprint, 2025.
\arxiv{2504.09837}.

\bibitem[TZ25]{TZ25}
Q.~Tang and T.~Zhang,
\emph{Sharp Schoenberg type inequalities and the de Bruin--Sharma problem},
preprint, 2025.
\arxiv{2508.10341}.

\bibitem[Tao26]{Tao26}
T.~Tao,
\emph{A digestion of the proof of Sendov's conjecture},
What's New, August~12, 2026,
\url{https://terrytao.wordpress.com/2026/08/12/a-digestion-of-the-proof-of-sendovs-conjecture/}.

\bibitem[Zha26a]{Zha26a}
T.~Zhang,
\emph{A refinement of Pawlowski's result},
Proc. Amer. Math. Soc. \textbf{154} (2026), no.~2, 775--782.
\doi{10.1090/proc/17471}.

\bibitem[Zha26b]{Zha26b}
T.~Zhang,
\emph{Beyond Sendov's conjecture: the quadratic Tang--Zhang inequality},
preprint, 2026.
\arxiv{2609.19126}.

\bibitem[Zha26c]{Zha26c}
T.~Zhang,
\emph{Distributions of zeros of a complex polynomial},
J. Math. Anal. Appl. \textbf{561} (2026), no.~1, Paper No.~130609.
\doi{10.1016/j.jmaa.2026.130609}.

\bibitem[Zha26d]{Zha26d}
T.~Zhang,
\emph{Weighted weak majorization and sharp Schoenberg-type inequalities},
preprint, 2026.
\url{https://zhangteng2000.github.io/files/weighted_weak_majorization.pdf}.

\bibitem[Zha27]{Zha27}
T.~Zhang,
\emph{When do the convex hulls of zeros and critical points coincide?},
J. Math. Anal. Appl. \textbf{566} (2027), no.~1, Paper No.~131002.
\doi{10.1016/j.jmaa.2026.131002}.


\bibitem[ZhaL]{ZhaL}
T.~Zhang,
\emph{A Lean~4 formalization of the Borcea 2-variance conjecture},
Lean~4 formalization with mathlib,
2026.
Available at
\url{https://github.com/zhangteng2000/Borcea-2-variance-conjecture-Lean-4}.


\bibitem[ZhaS]{ZhaS}
T.~Zhang,
\emph{Supplementary code and certificates for Borcea's $2$-variance conjecture},
supplementary material, version September~29, 2026.

\end{thebibliography}
\end{document}